\documentclass[reqno,12pt]{amsart}
\usepackage{amsmath,amsthm,amssymb,amsfonts}
\usepackage{graphicx}
\usepackage{amsmath,amsfonts,amssymb,amsthm,epsfig,color}
\usepackage{pgf,tikz}
\usepackage{amscd}
\usetikzlibrary{arrows}

\newtheorem{theorem}{Theorem}[section]

\newtheorem{lemma}[theorem]{Lemma}
\newtheorem{proposition}[theorem]{Proposition}
\newtheorem{remark}[theorem]{Remark}

\newtheorem{question}[theorem]{Question}

\theoremstyle{definition}
\newtheorem{definition}[theorem]{Definition}

\newcommand{\interior}[1]{{\rm int}(#1)}
\newcommand{\Mod}{\rm \mathcal M}
\newcommand{\MCG}{{\rm Mod}}

\newcommand{\R}{\mathbb{R}}
\newcommand{\N}{\mathbb{N}}

\newcommand{\matCP}{\mathbb{CP}}
\newcommand{\matH}{\mathbb{H}}
\newcommand{\Mno}{\mathcal M^{\rm no}}
\newcommand{\Mnon}{\mathcal M^{\rm no}}
\newcommand{\calM}{\mathcal M}
\newcommand{\Snon}{\mathcal S^{\rm no}}
\newcommand{\Hnon}{\mathcal H^{\rm no}}
\newcommand{\Hno}{\mathcal H^{\rm no}}
\newcommand{\calH}{\mathcal H}
\newcommand{\matR}{\mathbb{R}}
\newcommand{\matN}{\mathbb{N}}
\newcommand{\matZ}{\mathbb{Z}}

\newcommand{\Spy}{{\rm S}}

\usepackage{tikz,fp,ifthen,fullpage}
\usepackage{pgfplots}
\numberwithin{equation}{section}

\def\H{{\mathcal H}}

\def\Gi{{\Gamma_\infty}}
\def\calS{\mathcal{S}}

\author{Bruno Martelli}
\address{Dipartimento di Matematica ``Tonelli'', Largo Pontecorvo 5, 56127 Pisa, Italy}
\email{martelli at dm dot unipi dot it}

\author
{Matteo Novaga}
\address{Dipartimento di Matematica ``Tonelli'', Largo Pontecorvo 5, 56127 Pisa, Italy}
\email{novaga at dm dot unipi dot it}

\author{Alessandra Pluda}
\address{Dipartimento di Matematica ``Tonelli'', Largo Pontecorvo 5, 56127 Pisa, Italy}
\email{pluda at mail dot dm dot unipi dot it}

\author{Stefano Riolo}
\address{Dipartimento di Matematica ``Tonelli'', Largo Pontecorvo 5, 56127 Pisa, Italy}
\email{riolo at mail dot dm dot unipi dot it}

\title{Spines of minimal length}

\keywords{spine; length functional; constant curvature surfaces}
\subjclass[2010]{57M50; 49Q20; 53A10}

\begin{document}

\begin{abstract}
In this paper we raise the question whether every closed Riemannian manifold has a spine of minimal area, and we 
answer it affirmatively in the surface case.  
On constant curvature surfaces we introduce the \emph{spine systole}, a continuous real function on moduli space that measures the minimal length of a spine in each surface. We show that the spine systole is a proper function and has its global minima precisely on the extremal surfaces (those containing the biggest possible discs). 

We also study \emph{minimal spines}, which are critical points for the length functional. We completely classify minimal spines on flat tori, proving that the number of them is a proper function on moduli space. We also show that the number of minimal spines of uniformly bounded length is finite on hyperbolic surfaces.
\end{abstract}

\maketitle
\section{Introduction}
In this paper we study the spines of closed Riemannian manifolds that have minimal area with respect to the codimension-one Hausdorff measure. We describe the general setting, we prove that such spines exist on surfaces, and then we study the constant curvature case in detail.

\smallskip

In differential topology, a \emph{spine} of a closed smooth $n$-manifold $M$ is a smooth finite simplicial complex $P\subset M$ such that $M$ minus a small open ball collapses onto $P$. In particular $M\setminus P$ is an open ball. 

In all cases, we suppose that $\dim P < \dim M$, and this is the only restriction we make on dimensions: for instance, any point is a spine of the sphere $S^n$ for all $n\geq 1$. The polyhedron $P$ may also have strata of varying dimensions: for instance a natural spine of $S^2\times S^1$ is the union of the sphere $S^2 \times q$ and the circle $p \times S^1$.

A compact manifold $M$ has many different spines: as an example, one may give $M$ a Riemannian structure and construct $P$ as the cut locus of a point \cite{B}. The manifold $M$ has typically infinitely many pairwise non-homeomorphic spines, with portions of varying dimension. 

The notion of spine is widely employed in topology: for instance, it may be used to define a \emph{complexity} on manifolds \cite{Matveev, Matveev-book, M}, to study group actions \cite{A} and properties of Riemannian manifolds \cite{AB}. In dimensions 2 and 3 spines (with generic singularities) arise naturally and frequently as the dual of 1-vertex triangulations. Topologists usually consider spines only up to isotopy, and relate different spines (or triangulations) via some moves like ``flips'' on surfaces (see for instance \cite{AKP, FST, KP, STT}) and Matveev-Piergallini moves \cite{Matveev-moves, Piergallini} on 3-manifolds.

However, it seems that spines have not been much studied from a geometric measure theory point of view, and this is the main purpose of this paper.

\smallskip

If $M$ is a closed Riemannian manifold of dimension $n\geq 2$, every spine $P\subset M$ has a well-defined finite $(n-1)$-dimensional Hausdorff measure  $\H^{n-1}(P)$ called \emph{area}. For instance, a point in $S^n$ has zero area; the spine of $S^2\times S^1$ described above has the same area $4\pi$ of $S^2$. 

We are interested here in the following problem. 

\begin{question} \label{main:question}
Does every closed Riemannian manifold $M$ of dimension $n\geq 2$ have a spine of minimal area?
\end{question}

As an example, the answer is \emph{yes} for all spheres $S^n$ equipped with any Riemannian metric, since $S^n$ has a spine of zero area (a point). The reader may notice here that we need to allow smaller-dimensional spines like points to get a positive answer to Question \ref{main:question} on spheres (and also on other manifolds, see Section \ref{spines:subsection}).

This paper is essentially devoted to surfaces: their spines have dimension $\leq 1$ and it is of course more reasonable to employ the word \emph{length} to indicate their area. 


\begin{theorem}\label{existencein2d}
Every closed Riemannian surface $S$ has a spine $\Gamma$ of minimal length.
The spine $\Gamma$ is:
\begin{enumerate}
\item a point if $S$ is diffeomorphic to a sphere,
\item a closed geodesic if $S$ is diffeomorphic to a projective plane,
\item finitely many geodesic arcs meeting at trivalent points with angle $\frac 23 \pi$ otherwise.
\end{enumerate}
\end{theorem}


The theorem says everything about spheres, so we restrict our attention to the other surfaces $S$. 
In analogy with minimal surfaces, we say that a spine $\Gamma \subset S$ is \emph{minimal} 
if it is as prescribed in points (2) or (3) of the theorem: either a closed geodesic, or
finitely many geodesic arcs meeting with angle $\frac 23 \pi$ at trivalent points. 
These conditions are similar to that of having zero mean curvature for hypersurfaces. 
A spine of minimal length is minimal, but the converse may not hold: a minimal spine is a critical point of the length functional, 
while a spine of minimal length is a global minimum.

\smallskip

It is now natural to study these geometric objects on closed surfaces of constant curvature. Recall that all constant curvature metrics on a closed oriented surface $S$, considered up to orientation-preserving isometries and global rescalings, form the \emph{moduli space} $\Mod(S)$ of $S$. The moduli space is not compact: on the torus $T$, the space $\Mod(T)$ is the $(2,3,\infty)$ orbifold and is homeomorphic to a plane. 

We completely classify all minimal spines in flat tori. 

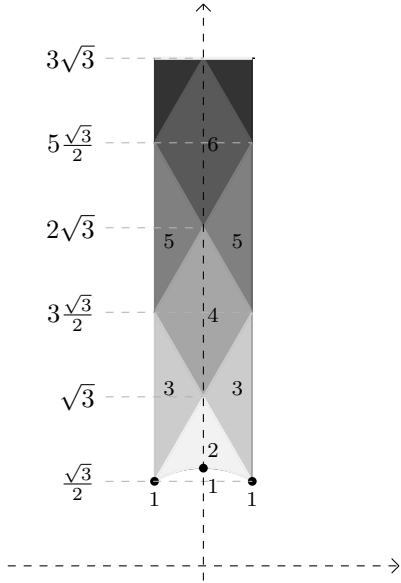
\begin{figure}
\begin{center}
\begin{tikzpicture}[scale=0.65]
\filldraw[fill=black!5!white]
(0,3.46) 
to[out= -120,in=60, looseness=1] (-1,1.73) -- (-1,1.73)
to[out=29, in=151, looseness=1] (1, 1.73) --
(1,1.73)to[out= 120,in=-60, looseness=1] (0,3.46); 
\filldraw[fill=black!20!white]
(0,3.46) 
to[out= -120,in=60, looseness=1] (-1,1.73) -- (-1,1.73)
to[out=90, in=-90, looseness=1] (-1,5.19) --
(-1,5.19) to[out= -60,in=120, looseness=1] (0,3.46); 
\filldraw[fill=black!20!white]
(0,3.46) 
to[out= -60,in=120, looseness=1] (1,1.73) -- (1,1.73)
to[out=90, in=-90, looseness=1] (1,5.19) --
(1,5.19) to[out= -120,in=60, looseness=1] (0,3.46); 
\filldraw[fill=black!35!white]
(0,6.92) 
to[out= -120,in=60, looseness=1] (-1,5.19) --
(-1,5.19)to[out= -60,in=120, looseness=1] (0,3.46)--
(0,3.46)to[out= 60,in=-120, looseness=1] (1,5.19)--
(1,5.19)to[out= 120,in=-60, looseness=1] (0,6.92); 
\filldraw[fill=black!50!white,shift={(0,3.46)}]
(0,3.46) 
to[out= -120,in=60, looseness=1] (-1,1.73) -- (-1,1.73)
to[out=90, in=-90, looseness=1] (-1,5.19) --
(-1,5.19) to[out= -60,in=120, looseness=1] (0,3.46); 
\filldraw[fill=black!50!white,shift={(0,3.46)}]
(0,3.46) 
to[out= -60,in=120, looseness=1] (1,1.73) -- (1,1.73)
to[out=90, in=-90, looseness=1] (1,5.19) --
(1,5.19) to[out= -120,in=60, looseness=1] (0,3.46); 
\filldraw[fill=black!65!white]
(0,10.39) 
to[out= -120,in=60, looseness=1] (-1,8.66) --
(-1,8.66)to[out= -60,in=120, looseness=1] (0,6.92)--
(0,6.92)to[out= 60,in=-120, looseness=1] (1,8.66)--
(1,8.66)to[out= 120,in=-60, looseness=1] (0,10.39);
\filldraw[fill=black!80!white]
(0,10.39) 
to[out= -120,in=60, looseness=1] (-1,8.66) -- (-1,8.66)
to[out=90, in=-90, looseness=1] (-1,10.39) --
(-1,10.39) to[out= 0,in=180, looseness=1] (0,10.39); 
\filldraw[fill=black!80!white]
(0,10.39) 
to[out= -60,in=120, looseness=1] (1,8.66) -- (1,8.66)
to[out=90, in=-90, looseness=1] (1,10.39) --
(1,10.39) to[out= 0,in=180, looseness=1] (0,10.39);  
\draw[scale=2,domain=-0.52: 0.52,
smooth,variable=\t,shift={(0,0)},rotate=0]plot({1*sin(\t r)},
{1*cos(\t r)}) ;
\draw[color=black!5!white,thick]
(-1,1.73)to[out=29, in=151, looseness=1] (1, 1.73)
(-1,1.73)to[out=60, in=-120, looseness=1](0,3.46)
(1,1.73)to[out=120, in=-60, looseness=1](0,3.46); 
\draw[color=black!20!white,thick]
(0,3.46)to[out=60, in=-120, looseness=1](1,5.19)
(0,3.46)to[out=120, in=-60, looseness=1](-1,5.19); 
\draw[color=black!20!white]
(1,1.73) to[out=90, in=-90, looseness=1](1,5.19)
(-1,1.73) to[out=90, in=-90, looseness=1](-1,5.19);
\draw[color=black!35!white,thick]
(-1,5.19)to[out=60, in=-120, looseness=1](0,6.92)
(1,5.19)to[out=120, in=-60, looseness=1](0,6.92);
\draw[color=black!50!white]
(1,5.19) to[out=90, in=-90, looseness=1](1,8.66)
(-1,5.19) to[out=90, in=-90, looseness=1](-1,8.66);
\draw[color=black!50!white,thick]
(0,6.92)to[out=60, in=-120, looseness=1](1,8.66)
(0,6.92)to[out=120, in=-60, looseness=1](-1,8.66);
\draw[color=black!65!white,thick]
(-1,8.66)to[out=60, in=-120, looseness=1](0,10.39)
(1,8.66)to[out=120, in=-60, looseness=1](0,10.39); 
\draw[color=black!80!white]
(1,8.66) to[out=90, in=-90, looseness=1](1,10.39)
(-1,8.66) to[out=90, in=-90, looseness=1](-1,10.39);
\draw[color=black!5!white, very thick]
(-1,10.39) to[out=0, in=180, looseness=1](1,10.39);

\path
(-1,1.73) node[circle,fill=black,scale=0.3] {}
(0,2) node[circle,fill=black,scale=0.3] {}
(1,1.73) node[circle,fill=black,scale=0.3] {};

\draw[color=black!30!white,dashed]
(-2,10.39)to[out=0, in=180, looseness=1](0,10.39)
(-2,8.66)to[out=0, in=180, looseness=1](1,8.66)
(-2,6.92)to[out=0, in=180, looseness=1](0,6.92)
(-2,5.19)to[out=0, in=180, looseness=1](1,5.19)
(-2,3.46)to[out=0, in=180, looseness=1](0,3.46)
(-2,1.73)to[out=0, in=180, looseness=1](1,1.73);
\path[font=\footnotesize]
(-2,10.39)node[left]{$3\sqrt{3}$}
(-2,8.66)node[left]{$5\frac{\sqrt{3}}{2}$}
(-2,6.92)node[left]{$2\sqrt{3}$}
(-2,5.19)node[left]{$3\frac{\sqrt{3}}{2}$}
(-2,3.46)node[left]{$\sqrt{3}$}
(-2,1.73)node[left]{$\frac{\sqrt{3}}{2}$};
\path[font=\tiny,color=black]
(1,1.73)node[below]{$1$}
(0.2,2)node[below]{$1$}
(-1,1.73)node[below]{$1$}
(0.2,2.73)node[below]{$2$}
(-0.7,4)node[below]{$3$}
(0.7,4)node[below]{$3$}
(0.2,5.5)node[below]{$4$}
(-0.7,7)node[below]{$5$}
(0.7,7)node[below]{$5$}
(0.2,9)node[below]{$6$};
\draw
(4,0) to[out= 135,in=-45, looseness=1] (3.85,0.15)
(4,0) to[out= -135,in=45, looseness=1] (3.85,-0.15)
(0,11.5) to[out= -45,in=135, looseness=1] (0.15,11.35)
(0,11.5) to[out= 45,in=-135, looseness=1] (-0.15,11.35);
\draw[dashed]
(0,-0.3) to[out= 90,in=-90, looseness=1] (0,11.5)
(-4,0) to[out= 0,in=-180, looseness=1] (4,0);
\end{tikzpicture}

\end{center}
\caption{The number of minimal spines on each oriented flat torus in moduli space $\calM(T)$, considered up to orientation-preserving isometries of $T$. The moduli space $\calM(T)$ is drawn via the usual fundamental domain in $\matH^2$. At each point $z\in\calM(T)$, the number $c(z)$ is the smallest among all numbers written on the adjacent strata (the function $c$ is lower semi-continuous).}
\label{c_intro:fig}
\end{figure}

\begin{theorem} \label{finite:spines:theorem}
Every oriented flat torus $T$ contains finitely many minimal spines up to orientation-preserving isometries. The number of such minimal spines is the proper function $c\colon \Mod(T) \to \N$ shown in Fig.~\ref{c_intro:fig}.
\end{theorem}

We see in particular that the square and the hexagonal torus at $z=i$ and $z=e^{\frac{\pi}{3}i}$ are the only flat tori that contains a unique minimal spine up to orientation-preserving isometries. We also discover that every flat torus $T$ contains finitely many minimal spines up to isometry, but this number increases (maybe unexpectedly?) and tends to $\infty$ as the flat metric tends to infinity in moduli space. 

By looking only at spines of minimal length we find the following.

\begin{proposition}
Every oriented flat torus $T$ has a unique spine of minimal length up to isometries of $T$. It also has a unique one up to orientation-preserving isometries, unless it is a rectangular non square torus and in this case it has two.
\end{proposition}

Let $\Spy\colon \calM(T) \to \matR$ be a function that assigns to a (unit-area) flat torus $z\in\calM(T)$ the minimal length $\Spy(z)$ of a spine in $T$. The function $\Spy$ may be called the \emph{spine systole} because it is analogous to the (geodesic) systole that measures the length of the shortest closed geodesic. We will prove that $\Spy$ has the following expression:
$$\Spy(z)=\sqrt{ \frac{1+|z|^2-|\Re(z)|+\sqrt{3}\Im(z)}{\Im(z)} }. $$
and has a unique global minimum at the hexagonal torus.

Now that the picture on the flat tori is perfectly understood, it is time to turn to higher genus hyperbolic surfaces and see if minimal spines have the same qualitative behavior there. Let $S_g$ be a closed orientable surface of genus $g \geq 2$ and let $\calM(S_g)$ be the moduli space of $S_g$. Every hyperbolic surface in $\calM(S_g)$ has a spine of minimum length by Theorem 
\ref{existencein2d} and its length defines a spine systole $\Spy\colon \calM(S_g) \to \matR$.

An \emph{extremal surface} is a hyperbolic surface that contains a disc of the maximum possible radius in genus $g$. Such surfaces were defined and studied by Bavard \cite{Ba} and various other authors, see \cite{GG} and the references therein. We can prove the following.

\begin{theorem} \label{hyperbolic:teo}
The function $\Spy\colon \calM(S_g) \to \matR$ is continuous and proper. Its global minima are precisely the extremal surfaces.
\end{theorem}

A full classification of all minimal spines on all hyperbolic surfaces would of course be desirable; for the moment, we content ourselves with the following.

\begin{proposition}
Every closed hyperbolic surface $S$ has finitely many minimal spines of bounded length.
\end{proposition}

In particular, we do not know if a closed hyperbolic surface has finitely many minimal spines overall (counting spines only up to isometry does not modify the problem, since the isometry group of $S$ is finite).

In higher dimensions, Question \ref{main:question} has already been addressed by Choe \cite{Choe} who has provided a positive answer for all closed irreducible 3-manifolds. The techniques used in that paper are much more elaborate than the ones we use here.

\section{Preliminary definitions}
\subsection{Spines.} \label{spines:subsection}
We recall some well-known notions that apply both in the piecewise-linear and in the smooth category of manifolds.

A \emph{smooth finite simplicial complex} (or a \emph{finite polyhedron}) in a smooth $n$-manifold $M$ with (possibly empty) boundary is a subset $P\subset \interior M$ homeomorphic to a simplicial triangulation, such that every simplex is diffeomorphic through some chart to a standard simplex in $\R^n$. 

The subset $P$ has a well-defined \emph{regular neighborhood}, unique up to isotopy: this is a piecewise-linear notion that also applies in the smooth category \cite{H}. A regular neighborhood $N$ of $P$ is a compact smooth $n$-dimensional sub-manifold $N\subset M$ containing $P$ in its interior which collapses simplicially onto $P$ for some smooth triangulation of $M$; as a consequence $N\setminus P$ is an open collar of $\partial N$. If the manifold $M$ itself is a regular neighborhood of $P$ we say that $P$ is a \emph{spine} of $M$.

This definition of course may apply only when $\partial M \neq \emptyset$, so to define a spine $P$ of a closed manifold $M$ we priorly remove a small open ball from $M$. In particular we get that $M\setminus P$ is an open ball in this case. 


In this paper we only consider spines $P$ of closed manifolds $M$. We also consider by assumption only spines $P$ with $\dim P < \dim M$, so that if $M$ is endowed with a Riemannian metric $P$ has a well-defined and finite $(n-1)$-dimensional Hausdorff measure $\H^{n-1}(P)$ called \emph{area}.

The area can be zero in some cases because we allow $P$ to have dimension strictly smaller than $n-1$. For instance, any point is a spine of $S^n$ and any complex hyperplane is a spine of complex projective space $\matCP^n$, because their complements are open balls: in both cases the spine has zero area because it has codimension bigger than 1.

A spine may have strata of mixed dimensions: for instance a natural spine for $S^m \times S^n$ is the transverse union of two spheres $S^m \times q$ and $p \times S^n$ (the reader may verify that the complement is an open ball).



In the following we consider only the case in which the dimension of the manifold $M$ is two.

\begin{remark}\label{Sos}
\rm{The dimension of a spine $\Gamma$ of a closed surface $S$ is less or equal one and its
homotopy type is completely determined  by that of $S$.
Indeed,
\begin{itemize}
\item $\Gamma$ is homotopically equivalent to a point if and only if $S$ is diffeomorphic to a sphere.
Indeed, a regular neighbourhood of a point must be a disc.
\item $\Gamma$ is homotopically equivalent to a circle if and only if $S$ is diffeomorphic to a projective plane.
Indeed, a regular neighbourhood of a circle could be nothing but a M\"obius strip,
because an annulus has too many boundary components. 
\item Finally, assume that $S$ is nor diffeomorphic to a sphere, neither to a projective plane.
A spine of $S$ is an embedded graph $\Gamma\subset S$. 
Denoting with $e$ the number of edges, with $v$ the number of vertices of $\Gamma$
and by $\chi$ the Euler characteristic, 
a necessary condition for $\Gamma$ to be a spine of $S$ is
\begin{equation}\label{Eulerchar}
\chi(S)-1=\chi (S-B^2)=  \chi(\Gamma)=v-e\,.
\end{equation}
\end{itemize}}
\end{remark}

\subsection{Networks in Riemannian surfaces.}
Let $S$ be a Riemannian surface.
An  embedded graph $\Gamma\subset S$ is a 
\emph{network}:
a union of a finite number of supports of simple smooth curves
$\gamma^i:\left[0,1\right]\rightarrow S$, intersecting only at their endpoints.
A point in which two or more curves concur is called \emph{multipoint}.
Each curve $\gamma^i$ of the network has length $\H^1\left( {\rm Im} \left(\gamma^i\right) \right)$
and the length of $\Gamma$ is the sum of the lengths of all the curves, that is,
\begin{equation}\label{lengthfunctional}
L\left(\Gamma\right)=\H^1\left(\bigcup_{i=1}^n\left( {\rm Im} \left(\gamma^i\right) \right) \right)
=\sum_{i=1}^n \H^1\left( {\rm Im} \left(\gamma^i\right) \right) .
\end{equation}

In the following, we will search the minima of this functional, 
restricted to the set of spines of a closed Riemannian surface $S$, endowed with the Hausdorff topology.

\subsection{First variation of the length functional.}
Let $\Gamma$ be a network in a closed Riemannian surface $S$.
Let ${\Phi_t}$ with $t\in \left[0,T\right]$ be a smooth 
family of diffeomorphisms of $S$, with $\Phi_0(x)=x$ for all $x\in\Gamma$
and let $\Gamma_t=\Phi_t\left(\Gamma\right)$.
Consider the vector field $X$ on $S$ defined as
$$X(x)=\frac{d}{dt}\Phi_t(x)\Big\vert_{t=0}.$$

The first variation formula for the length functional $L$ 
for a network $\Gamma$ is (see for instance~\cite{maggi}):
\begin{equation}\label{firstvar}
\frac{d}{dt}L(\Gamma_t)\Big\vert_{t=0}
=\int_{\Gamma}{\rm div}_{\tau} X\,d\H^1\,,
\end{equation}
where with ${\rm div}_{\tau}$ we denote the tangential divergence.

Let $x_1,\ldots, x_m$ be the multipoints of $\Gamma$.
At every multipoint $x_j$, $l_j$ curves  concur.  For $k=1,\ldots l_j$
denote with
$\tau_j^k$  the unit tangent vector to each of that curves at $x_j$. 
In particular, 
\begin{equation}\label{1}
\frac{d}{dt}L(\Gamma_t)\Big\vert_{t=0}
=-\int_{\Gamma} H\cdot X\,d\H^1
+\sum_{j=1}^m \left(  \sum_{k=1}^{l_j}\tau_j^k\cdot X\right) \,,
\end{equation}
where $H$ is the curvature of the curve.
A network 
$\Gamma$ is \emph{stationary}
if there holds
\begin{equation}\label{statnet}
\int_{\Gamma}{\rm div}_{\tau} X\,d\H^1=0\,\quad\text{for every vector field}\, X\, \text{on} \, S\,,
\end{equation}

Thanks to~\eqref{1}, condition~\eqref{statnet} is equivalent to require that each curve of the network $\Gamma$ 
is a geodesic arc and the sum 
of the unit tangent vectors of concurring curves at a common endpoint is equal to zero.

We will actually consider only spines, and reserve the word \emph{minimal} for a more restrictive configuration where only three edges concur at each multipoint, see Definition \ref{minimal:definition}. The reason for that is that spines of minimal length will always be of this type.

\subsection{Rectifiable sets and stationary varifolds.}
Let $S$ be a Riemannian surface.
We recall that a set $\Gamma\subset S$ is \emph{countably $1$-rectifiable} if it can be covered by countably many
images of Lipschitz maps from $\mathbb{R}$ to $S$, except a subset $\H^1-$negligible.

Let $\Gamma$ be a countably $1$-rectifiable, $\H^1$-measurable subset of $S$
and let $\theta$ be a positive locally $\H^1$-integrable function on $\Gamma$.
Following \cite{Allard,AA}
we define the \emph{rectifiable $1$-varifold} $(\Gamma,\theta)$
to be the equivalence class of all pairs $(\widetilde{\Gamma},\widetilde{\theta})$,
where $\widetilde{\Gamma}$ is countably $1$-rectifiable with 
$\H^1\left((\Gamma\setminus\widetilde{\Gamma})\cup(\widetilde{\Gamma}\setminus\Gamma)\right)=0$
and where $\widetilde{\theta}=\theta$ $\H^1$-a.e. on $\Gamma\cap\widetilde{\Gamma}$.
The function $\theta$ is called \emph{multiplicity} of the varifold.

A varifold $(\Gamma,\theta)$ is \emph{stationary} if there holds
\begin{equation}\label{statvar}
\int_\Gamma \theta \,{\rm div}_\tau X\,d\H^1 = 0,
\end{equation}
for any vector field $X$ on $S$. Notice that, if $\theta$ is constant and $\Gamma$ is a network,
condition \eqref{statvar} is consistent with \eqref{statnet}.   Recalling \eqref{firstvar}
this means that the network $\Gamma$ is a critical point of the length functional.

\section{Riemannian surfaces}

In this section we prove an existence result 
for the spines of minimal length of a closed Riemannian surface
(Theorem~\ref{existencein2d}).

\begin{proposition}
Consider a closed Riemannian surface $S$ not diffeomorphic to a sphere.
Then there is a constant $K> 0$ such that $L(\Gamma)\geq K$ for every spine $\Gamma$ of $S$.
\end{proposition}
\begin{proof}
Let $r>0$ be the injectivity radius of $S$. 
Every homotopically non-trivial closed curve in $S$ has length at least $2r$. 
Every spine $\Gamma$ contains at least one homotopically non-trivial 
embedded closed curve and hence has length at least $K=2r$.

To verify the last fact, recall that $S\setminus \Gamma$ is an open $2$-disc 
and hence the inclusion map $i\colon \Gamma \hookrightarrow S$ 
induces a surjection $i_*\colon \pi_1(\Gamma)\to \pi_1(S)$. 
Since $\pi_1(S)\neq \{e\}$ 
the spine $\Gamma$ contains some homotopically non-trivial loop, 
and this easily implies that it also contains an embedded one.
\end{proof}

{\it Proof of Theorem~\ref{existencein2d}.\;}
Thanks to the topological observations made in Subsection~\ref{Sos}, the case of the sphere is
trivial and that of the projective plane is well known.
Therefore, we will henceforth suppose that $S$ is neither diffeomorphic to a sphere nor to a projective plane.
We want to prove that
\begin{equation} \label{ell:eqn}
\ell = \inf\big\{L(\Gamma)\ \vert\ \Gamma\,\text{is a spine of}\,S\big\}\,.
\end{equation}
is a minimum. We subdivide the proof into four steps. \\

\textbf{Step 1:} 
\textit{$\Gamma_n$ minimizing sequence of spines converges 
to a closed connected and rectifiable set $\Gamma_\infty$.}

We have just defined 
$$\ell = \inf_{\Gamma\subset S} \H^1(\Gamma)$$
where $\Gamma$ varies among all spines of $S$. 
Let $\Gamma_n$ be a \emph{minimizing sequence} of spines, that is a sequence 
such that $\H^1(\Gamma_n) \to \ell$. In particular,
the sets $\Gamma_n$ are closed, connected, rectifiable, and 
$S\setminus \Gamma_n$ is homeomorphic to an open disc for every $n$.

Thanks to Blaschke Theorem \cite[Theorem~4.4.15]{AT},
up to passing to a subsequence, the sequence converges 
$\Gamma_n\to \Gi$ to a compact set $\Gi$ in the Hausdorff distance, and
by Golab Theorem \cite[Theorem~4.4.17]{AT} we get
\begin{equation}\label{convga}
\H^1(\Gi)\,\le\,\liminf_n\H^1(\Gamma_n) = \ell
\end{equation}
and $\Gi$ is connected.
Moreover by the Rectifiability Theorem \cite[Theorems 4.4.7]{AT},
the limit set $\Gi$ is rectifiable and connected 
by injective rectifiable curves.\\ 
\textbf{Step 2:}
\textit{Structure of $\Gamma_\infty$.}

We have proved that $\Gamma_\infty$ has minimal length $\ell$.
We can associate to $\Gamma_\infty$ a rectifiable varifold with multiplicity $1$.

Since $\Gamma_\infty$ is of minimal length, it follows that the corresponding 
varifold is stationary, hence is composed by a finite number of geodesic segments, 
joining finitely many nodes (see \cite{AA}). The minimality of $\Gamma_\infty$ 
also implies that each node is the meeting point of exactly three curves,
forming angles of $\frac{2\pi}{3}$ by standard arguments:
suppose by contradiction that more than three curves concur at meeting point $O$.
Consider a sufficiently small neighbourhood of this multiple point $O$ where there are not other nodes.
Take two segments concurring in $O$ forming an angle $<\frac{2\pi}{3}$
and two points $P$ and $Q$ of these segments in the considered neighbourhood.
Lift $P$ and $Q$ to the tangent plane of the surface at $O$ through the exponential map.
We know that in the tangent plane the Steiner configuration is the minimal length configuration 
joining $O$, $P$ and $Q$.
Replace the part of  two curves on the surface joining $P$ and $Q$ with $O$ in the neighbourhood
with the image through the exponential map of the three segments of the Steiner configuration.
Repeating iteratively this procedure, we obtain only triple junctions.\\
\textbf{Step 3:}
\textit{The set $\Gi$ intersects every homotopically non-trivial closed curve $\gamma$.}

If $\gamma$ is a homotopically non-trivial closed curve, then
each $\Gamma_n$ intersects $\gamma$, 
because $\gamma$ cannot be contained in the open disc $S\setminus \Gamma_n$. 
Hence, the Hausdorff limit $\Gi$ also intersects $\gamma$ because $\gamma$ is compact.\\
\textbf{Step 4:}
\textit{The set $\Gi$ is a spine.}

Let $\calS$ be the set of all closed subsets of $\Gi$ 
that intersect every homotopically non-trivial closed curve in $S$. 
By Zorn's lemma there is a $\Gamma \in \calS$ which is minimal with respect to inclusion. 
We now prove that $\Gamma$ is a spine.

The open set $S\setminus \Gamma$ contains no non-trivial closed curve. 
If $S\setminus \Gamma$ is an open disc we are done: we prove that this is the case.
If one component $U$ of $S\setminus \Gamma$ is not an open disc, it is not simply connected: 
hence it contains a simple closed curve $\gamma$ which is homotopically non-trivial in $U$, 
but is necessarily trivial in $S$ since it does not intersect $\Gamma$. 
Therefore $\gamma$ bounds a disc $D\subset S$. 
Since $D\not\subset U$, the intersection $D\cap \Gamma$ is non-empty: 
if we remove $D\cap \Gamma$ from $\Gamma$ we obtain another element of $\calS$ 
strictly contained in $\Gamma$, a contradiction.

Therefore $S\setminus \Gamma$ consists of open discs only. If they are at least two, 
there is an arc in $\Gamma$ adjacent to two of them: by removing 
from $\Gamma$ an open sub-arc in this arc we get again another element in $\calS$ strictly contained in $\Gamma$.

We know that $\Gamma\subset \Gamma_\infty$ is a spine, and hence $\Gamma = \Gamma_\infty$, for if not the length of $\Gamma$ would be strictly smaller than that $\ell$ of $\Gamma_\infty$, a contradiction by (\ref{ell:eqn}). 
\qed

\begin{definition} \label{minimal:definition}
A spine of a Riemannian surface $S$ is \emph{minimal} if it is a point, a closed geodesic, or if it is composed by finitely many geodesic arcs, meeting with angle $\frac{2\pi}{3}$ at trivalent points.
\end{definition}
We have shown that a spine of minimal length is minimal. Of course, the converse may not hold.
However, a minimal spine is a stationary point of the length functional thanks to~\eqref{1}.

\begin{remark}\rm
Notice that the extension of the existence result in higher dimension present several difficulties: 
there is no higher dimensional version of Golab Theorem, 
because of the lack of semicontinuity of the Hausdorff measure.
Also, it is not clear if a limit of spines is still a spine. 
However, as we already observed, the existence of a spine of minimal area 
in a closed irreducible 3-manifold has been proved in \cite{Choe}.
\end{remark}

\begin{remark}\label{numberofvertices}\rm{
If the surface $S$ is neither diffeomorphic to a sphere nor to a projective plane, we have shown
that minimal spines are trivalent graphs. Hence, adding the 
equation $3v=2e$ to~\eqref{Eulerchar}, we get that the number of edges and that of vertices of a minimal spine
are completely determined by the topology of $S$.}
\end{remark}

\subsection{Non positive constant curvature surfaces.}
We restrict the attention to the case of non positive constant curvature surfaces. 
The goal is to show that minimal spines are local minimizers
for the length functional, justifying our choice for the adjective ``minimal''.
Let us begin with a definition and a well-known lemma about the convexity of the distance function in $\matH^2$ that we take from \cite{Fm}.
\begin{definition}\label{convexcomb}
Let $x_1,x_2$ be points in $\mathbb{H}^2$, $\mathbb{R}^2$ or $S^2$
and $\lambda\in[0,1]$.
The \emph{convex combination} $x=\lambda x_1+(1-\lambda)x_2$
is defined as follows:
\begin{align*}
\text{in}\;\mathbb{R}^2:\quad\quad & x=\lambda x_1+(1-\lambda)x_2\\
\text{in}\;\mathbb{H}^2,S^2:\quad\quad & 
x=\frac{\lambda x_1+(1-\lambda)x_2}{\Vert \lambda x_1+(1-\lambda)x_2 \Vert}
\end{align*}
where in the $\mathbb{H}^2$ case we are considering the hyperboloid model in $\mathbb{R}^3$
with the Lorentzian scalar product  $\left\langle\cdot,\cdot\right\rangle$
and $\Vert v \Vert=\sqrt{-\left\langle v,v \right\rangle}$.
\end{definition}

\begin{lemma}\label{convexity}
Let $x_1,x_2,y_1,y_2$ be points of the hyperbolic plane $\mathbb{H}^2$.
For $\lambda\in (0,1)$, consider the convex combinations
$x_\lambda=\lambda x_1+(1-\lambda)x_2$ and $y_\lambda=\lambda y_1+(1-\lambda)y_2$.
Then, we have
\begin{equation*}
d(x_\lambda,y_\lambda)\leq\lambda d(x_1,y_1)+(1-\lambda)d(x_2,y_2)\,,
\end{equation*}
with equality only if $x_1,x_2,y_1,y_2$ belong to the same line. 
Here $d$ denotes the distance in $\mathbb H^2$.
\end{lemma}
\begin{proof}
Without loss of generality, for simplicity, we prove only the case $\lambda=\frac 12$,
therefore $x_\lambda$ (resp. $y_\lambda$) is the midpoint of $x_1$ and $x_2$ (resp. $y_1$ and $y_2$).
If $x_\lambda = y_\lambda$ the theorem is trivial, hence we suppose $x_\lambda \neq y_\lambda$.

Let $\sigma_p$ be the reflection at the point $p\in \mathbb{H}^2$. 
The map $\tau = \sigma_{y_\lambda} \circ \,\sigma_{x_\lambda}$
translates the line $r$ containing the segment $x_\lambda y_\lambda$ by the quantity $2d(x_\lambda, y_\lambda)$:
hence it is a hyperbolic transformation with axis $r$. 
We call $z_i =\tau(x_i)$ and note that $z_1 = \sigma_{y_\lambda}(x_2)$, 
hence $d(x_2, y_2) = d(z_1, y_1)$.

\begin{figure}[h]
\begin{tikzpicture}[scale=0.75]
\draw[scale=3,domain=-1.1: 1.1,
smooth,variable=\t,shift={(0,0)},rotate=0]plot({1*sin(\t r)},
{1*cos(\t r)}) ;
\draw[scale=7.6,domain=0.4: 1,
smooth,variable=\t,shift={(-1,-0.34)},rotate=0]plot({1*sin(\t r)},
{1*cos(\t r)}) ;
\draw[scale=7.7,domain=-0.89: 0,
smooth,variable=\t,shift={(0.57,-0.43)},rotate=0]plot({1*sin(\t r)},
{1*cos(\t r)}) ;
\draw[scale=1,domain=-0.2: 0.9,
smooth,variable=\t,shift={(-2.38,-1.44)},rotate=0]plot({5*sin(\t r)},
{5*cos(\t r)}) ;
\draw[scale=3,domain=-1.3: -1.1,
smooth,variable=\t,shift={(0,0)},rotate=0,dashed]plot({1*sin(\t r)},
{1*cos(\t r)}) ;
\draw[scale=7.6,domain=0.43: 0.3,
smooth,variable=\t,shift={(-1,-0.34)},rotate=0,dashed]plot({1*sin(\t r)},
{1*cos(\t r)}) ;
\draw[scale=7.7,domain=-0.99: -0.89,
smooth,variable=\t,shift={(0.57,-0.43)},rotate=0,dashed]plot({1*sin(\t r)},
{1*cos(\t r)}) ;
\draw[scale=1,domain=-0.4: -0.2,
smooth,variable=\t,shift={(-2.38,-1.44)},rotate=0,dashed]plot({5*sin(\t r)},
{5*cos(\t r)}) ;
\draw[scale=3,domain=1.1: 1.3,
smooth,variable=\t,shift={(0,0)},rotate=0,dashed]plot({1*sin(\t r)},
{1*cos(\t r)}) ;
\draw[scale=7.6,domain=0.9: 1.1,
smooth,variable=\t,shift={(-1,-0.34)},rotate=0,dashed]plot({1*sin(\t r)},
{1*cos(\t r)}) ;
\draw[scale=7.7,domain=0: 0.2,
smooth,variable=\t,shift={(0.57,-0.43)},rotate=0,dashed]plot({1*sin(\t r)},
{1*cos(\t r)}) ;
\draw[scale=1,domain=0.9: 1.1,
smooth,variable=\t,shift={(-2.38,-1.44)},rotate=0,dashed]plot({5*sin(\t r)},
{5*cos(\t r)}) ;
\draw
(0,1.3) to[out= 90,in=-90, looseness=1](0,5.5);
\draw[dashed]
(0,5.5) to[out= 90,in=-90, looseness=1](0,6);
\draw[dashed]
(0,0.75) to[out= 90,in=-90, looseness=1](0,1.3);
\path
(1,2.2) node[circle,fill=black,scale=0.5] {}
(3.6,4.35) node[circle,fill=black,scale=0.5] {}
(-1.35,1.8) node[circle,fill=black,scale=0.5] {}
(-3,3.5) node[circle,fill=black,scale=0.5] {}
(1.85,2.35) node[circle,fill=black,scale=0.5] {}
(-1.85,2.35) node[circle,fill=black,scale=0.5] {}
(0,2.25) node[circle,fill=black,scale=0.5] {}
(0,3) node[circle,fill=black,scale=0.5] {}
(0,4.5) node[circle,fill=black,scale=0.5] {};
\path[font=\normalsize] 
(3.6,4.35) node[below] {$z_2$}
(1,2.2) node[below] {$z_1$}
(-3,3.5) node[above] {$x_1$}
(-1.35,1.8) node[left] {$x_2$}
(1.55,2.65) node[right] {$z_\lambda$}
(-1.55,2.65) node[left] {$x_\lambda$}
(0,4.7) node[right] {$y_1$}
(0,3.4) node[right] {$y_\lambda$}
(0,2.45) node[right] {$y_2$};
\end{tikzpicture}
\caption{}
\end{figure}

 The triangular inequality implies that
\begin{equation}\label{equno}
d(x_1, z_1)\leq d(x_1, y_1) + d(y_1, z_1) = d(x_1, y_1) + d(x_2, y_2)\,.
\end{equation}
We notice that the equality holds only if $x_1, y_1$ and $z_1$ belong all to the same line.

A hyperbolic transformation has minimum displacement on its axis $r$, hence
\begin{equation}\label{eqdue}
2d(x_\lambda, y_\lambda) = d(x_\lambda, z_\lambda) = d(x_\lambda, \tau(x_\lambda)) \leq d(x_1, \tau(x_1)) = d(x_1, z_1)\,,
\end{equation}
and the equality holds only if $x_1$ (and hence $x_2$) is in $r$.
Finally we get $d(x_\lambda, y_\lambda) \leq \frac12 (d(x_1, y_1) + d(x_2, y_2))$ and hence $d$ is convex.

Notice that the equality holds both in \eqref{equno} and in \eqref{eqdue}
only if $x_1,x_2,y_1,y_2$ belong to the same line $r$.
\end{proof}

\begin{theorem}\label{minima}
Minimal spines of closed surfaces of non positive constant curvature
are local minima for the length functional among spines, with 
respect to the Hausdorff distance.

If the curvature is negative, these are strict local minima.
\end{theorem}

\begin{proof}
We prove the first statement by contradiction.

Consider a sequence of spines $\Gamma_n$ of the surface $S$
converging in the Hausdorff distance to a minimal spine 
$\Gamma$, such that $L(\Gamma_n)< L(\Gamma)$.
The minimal spine $\Gamma$ of $S$ is a network composed by geodesic arcs
joining $k$ triple junctions $x_1,\,_{\cdot\cdot\cdot}\,,\,x_k$, 
where the number $k$ depends only on the topology of $S$
(see Remark~\ref{numberofvertices}).
For $n$ big enough, also $\Gamma_n$ have $k$
triple junctions $x_{1,n},\,_{\cdot\cdot\cdot}\,,\,x_{k,n}$.
Moreover, we can suppose that for $n$ big enough $\Gamma_n$ are composed only by geodesic segments.
Indeed if $\Gamma_n$ are not composed by geodesic segments, we can replace $\Gamma_n$ with 
$\widetilde{\Gamma}_n$, union of geodesic arcs, with the same triple junctions of $\Gamma_n$, 
and the value of the length functional decreases
$L(\widetilde{\Gamma}_n)\leq L(\Gamma_n)< L(\Gamma)$.

For $n$ big enough, and for every $\lambda\in [0,1]$ and $i\in\{1,\ldots k\}$, take the convex combination
$x_{i,n}^\lambda=(1-\lambda)x_i+\lambda x_{i,n}$
and define $\Gamma_n^\lambda$ as the spine
obtained by joining the points $x_{i,n}^\lambda$ with geodesic segments
in the same pattern of $\Gamma$.
We get a continuous family of spines
$\{\Gamma_n^\lambda\}_{\lambda\in[0,1]}$ such that
$\Gamma_n^1=\Gamma_n$ and $\Gamma_n^0=\Gamma$.
By Lemma~\ref{convexity}, the continuous function
$F_n(\lambda)=L(\Gamma_n^\lambda)$ is convex (convexity of the distance function is easily proved also in the Euclidean case)
and $F_n(1)\leq F_n(0)$.
We also have $F'_n(0)=0$ because the minimal spine
$\Gamma$ is a stationary point of the length functional.
This  implies that 
$L(\Gamma_n)=F_n(1)\geq F_n(0)=L(\Gamma)$
and we have a contradiction.

In the hyperbolic case, Lemma \ref{convexity} provides strict convexity and hence $\Gamma$ is a strict local minimum.
\end{proof}

\begin{remark}\rm{
It is not restrictive to consider local minima of the length functional only among spines
and not in the larger class of networks.
Indeed, if we take a smooth family 
${\Phi_t}$ of diffeomorphism of $S$ with $t\in \left[0,T\right]$
and a spine $\Gamma$ and we consider a small perturbation 
of $\Gamma$ via these diffeomorphisms, $\Gamma_t=\Phi_t\left(\Gamma\right)$
is still a spine for $t$ small enough.
In particular, in the proof of Theorem~\ref{minima}, 
we show that $L(\Gamma)\leq L(\Phi_t\left(\Gamma\right))$, for all ${\Phi_t}$ and for $t$ small enough.}
\end{remark}

\section{Flat tori}
In this section, we analyse minimal spines of the closed surface of genus $1$:
the torus $T=S^1 \times S^1$. In particular, we will fully determine all the minimal spines on $T$, endowed with any Euclidean metric.
\subsection{Minimal spines of Riemannian tori.}
From Remark~\ref{numberofvertices}, for any Riemannian metric on the torus $T$, minimal spines have exactly $2$ vertices and $3$ edges.
There are only two kinds of graph satisfying these properties:
the $\theta$-graph and the eyeglasses (see Figure \ref{graftriv}).

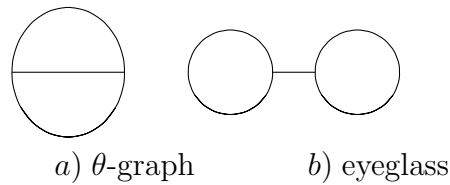
\begin{figure}[h]
\begin{tikzpicture}[scale=0.75]
\draw[color=black,scale=1,domain=-4.141: 4.141,
smooth,variable=\t,shift={(-1,0)},rotate=0]plot({1*sin(\t r)},
{1.15*cos(\t r)}) ;
\draw[color=black,scale=0.75,domain=-4.141: 4.141,
smooth,variable=\t,shift={(2.5,0)},rotate=0]plot({1*sin(\t r)},
{1*cos(\t r)}) ;
\draw[color=black,scale=0.75,domain=-4.141: 4.141,
smooth,variable=\t,shift={(5.5,0)},rotate=0]plot({1*sin(\t r)},
{1*cos(\t r)}) ;
\draw
(-2,0) to[out= 0,in=180, looseness=1] (0,0);
\draw[scale=0.75]
(3.5,0) to[out= 0,in=180, looseness=1] (4.5,0);
\path[font=\normalsize]
(0,-1.2) node[below]{$a)\;\theta\text{-graph}$}
(4.5,-1.2) node[below]{$b)\;\text{eyeglass}$};
\end{tikzpicture}
\caption{The two trivalent graphs with three edges.}
\label{graftriv}
\end{figure}

Both graphs can be embedded in a torus, but only the first in a way to obtain a spine.
Indeed, it is easy to find a $\theta$-spine in the torus (see Figure~\ref{graph}-$a$),
and actually infinitely many non isotopic ones.
Consider, instead, the eyeglasses as an abstract graph.
Thickening the interiors of its edges, we get three bands.
To have a spine on $T$, it remains to attach them, getting
a surface homeomorphic to $T-B^2$.
This is impossible.
The band corresponding to one of the two ``lenses'' of the eyeglasses
has to be glued to itself.
The are only two ways to do this: one would give
an unorientable surface (see Figure~\ref{graph}-$b$),
the other would have too many boundary components
(see Figure~\ref{graph}-$c$).

We have shown the following

\begin{figure}[h!]
\begin{center}

\begin{minipage}[c]{.4\textwidth}
\centering
\begin{center}
\begin{tikzpicture}[line cap=round,line join=round,>=triangle 45,x=1.0cm,y=1.0cm,scale=0.2]
\clip(-1.89,-7.73) rectangle (23.2,7.05);
\draw [shift={(6.11,0.1)},line width=0.6pt]  plot[domain=-0.48:4.8,variable=\t]({1*2.9*cos(\t r)+0*2.9*sin(\t r)},{0*2.9*cos(\t r)+1*2.9*sin(\t r)});
\draw [shift={(6.11,0.1)},line width=0.6pt]  plot[domain=-0.54:5.18,variable=\t]({1*5.67*cos(\t r)+0*5.67*sin(\t r)},{0*5.67*cos(\t r)+1*5.67*sin(\t r)});
\draw [shift={(11.25,0.1)},line width=0.6pt]  plot[domain=2.66:3.63,variable=\t]({1*2.9*cos(\t r)+0*2.9*sin(\t r)},{0*2.9*cos(\t r)+1*2.9*sin(\t r)});
\draw [shift={(11.25,0.1)},line width=0.6pt]  plot[domain=2.61:3.68,variable=\t]({1*5.67*cos(\t r)+0*5.67*sin(\t r)},{0*5.67*cos(\t r)+1*5.67*sin(\t r)});
\draw [shift={(11.25,0.1)},line width=0.6pt]  plot[domain=-2.04:2.04,variable=\t]({1*5.67*cos(\t r)+0*5.67*sin(\t r)},{0*5.67*cos(\t r)+1*5.67*sin(\t r)});
\draw [shift={(11.25,0.1)},line width=0.6pt]  plot[domain=-1.66:1.66,variable=\t]({1*2.9*cos(\t r)+0*2.9*sin(\t r)},{0*2.9*cos(\t r)+1*2.9*sin(\t r)});
\draw [line width=1.4pt] (7.91,-2.55)-- (9.36,-2.57);
\draw [line width=1.4pt] (7.91,-2.55)-- (6.76,-4.05);
\draw [line width=1.4pt] (9.36,-2.57)-- (10.42,-4.08);
\draw [shift={(6.11,0.1)},line width=1.4pt]  plot[domain=-0.69:4.87,variable=\t]({1*4.2*cos(\t r)+0*4.2*sin(\t r)},{0*4.2*cos(\t r)+1*4.2*sin(\t r)});
\draw [shift={(11.25,0.1)},line width=1.4pt]  plot[domain=2.54:3.81,variable=\t]({1*4.26*cos(\t r)+0*4.26*sin(\t r)},{0*4.26*cos(\t r)+1*4.26*sin(\t r)});
\draw [shift={(11.25,0.1)},line width=1.4pt]  plot[domain=-1.77:1.86,variable=\t]({1*4.26*cos(\t r)+0*4.26*sin(\t r)},{0*4.26*cos(\t r)+1*4.26*sin(\t r)});
\begin{scriptsize}
\fill [color=black] (9.36,-2.57) circle (2.0pt);
\fill [color=black] (7.91,-2.55) circle (2.0pt);
\end{scriptsize}
\path[font=\normalsize]
(-1,-2.5) node[below]{$a)$};
\end{tikzpicture}
\end{center}
\end{minipage}
\hspace{10mm}
\begin{minipage}[c]{.4\textwidth}
\centering
\begin{center}
\begin{tikzpicture}[rotate=-90,scale=0.9]
\draw[color=black,scale=0.75,domain=-4.141: 4.141,
smooth,variable=\t,shift={(2.5,0)},rotate=0]plot({1*sin(\t r)},
{1*cos(\t r)}) ;
\draw[scale=0.75]
(3.5,0) to[out= 0,in=180, looseness=1] (4,0);
\draw[scale=0.75,dashed]
(4,0) to[out= 0,in=180, looseness=1] (4.7,0);
\draw[color=black,scale=0.5,domain=-2.5: 3.2,
smooth,variable=\t,shift={(3.75,0)},rotate=0,thick]plot({1*sin(\t r)},
{1*cos(\t r)}) ;
\draw[color=black,domain=-4.51: -3,
smooth,variable=\t,shift={(1.85,0)},rotate=0,thick]plot({1*sin(\t r)},
{1*cos(\t r)}) ;
\draw[color=black,domain=-2.5: 1.41,
smooth,variable=\t,shift={(1.85,0)},rotate=0,thick]plot({1*sin(\t r)},
{1*cos(\t r)}) ;
\draw[scale=0.75,thick]
(1.6,-1) to[out= -50,in=180, looseness=1] (2.5,-0.65);
\draw[scale=0.75,thick]
(2.3,-1.32) to[out=150,in=60, looseness=1] (2.15,-1.05);
\draw[scale=0.75,thick]
(2.15,-0.7) to[out=150,in=-30, looseness=1] (2.1,-0.53);
\draw[scale=0.75,thick]
(3.75,0.2) to[out= 0,in=180, looseness=1] (4.7,0.2);
\draw[scale=0.75,thick]
(3.75,-0.2) to[out= 0,in=180, looseness=1] (4.7,-0.2);
\path[font=\normalsize]
(2.75,-1.5) node[below]{$b)$}
(2.75,1.25) node[below]{$c)$};

\draw[color=black,scale=0.75,domain=-4.141: 4.141,
smooth,variable=\t,shift={(2.5,3.5)},rotate=0]plot({1*sin(\t r)},
{1*cos(\t r)}) ;
\draw[scale=0.75,shift={(0,3.5)}]
(3.5,0) to[out= 0,in=180, looseness=1] (4,0);
\draw[scale=0.75,dashed,shift={(0,3.5)}]
(4,0) to[out= 0,in=180, looseness=1] (4.7,0);
\draw[color=black,scale=0.5,domain=-4.141: 4.141,
smooth,variable=\t,shift={(3.75,5.25)},rotate=0,thick]plot({1*sin(\t r)},
{1*cos(\t r)}) ;
\draw[color=black,domain=-4.51: 1.41,
smooth,variable=\t,shift={(1.85,2.65)},rotate=0,thick]plot({1*sin(\t r)},
{1*cos(\t r)}) ;
\draw[scale=0.75,shift={(0,3.5)},thick]
(3.75,0.2) to[out= 0,in=180, looseness=1] (4.7,0.2);
\draw[scale=0.75,shift={(0,3.5)},thick]
(3.75,-0.2) to[out= 0,in=180, looseness=1] (4.7,-0.2);
\end{tikzpicture}
\end{center}
\end{minipage}
\end{center}
\caption{$a)$ A regular neighbourhood of a $\theta$-spine on $T$. $b)-c)$ An eyeglasses spine on $T$ does not exist.}
\label{graph}
\end{figure}
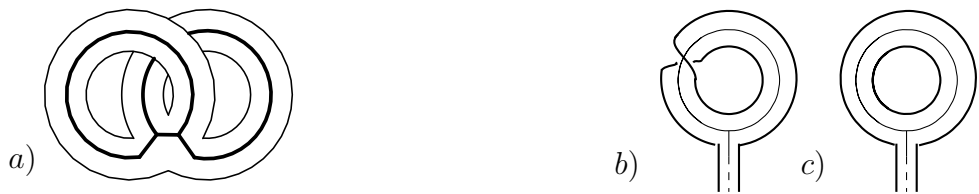

\begin{proposition}
Trivalent spines on the torus are $\theta$-graphs.
\end{proposition}
Combined with Theorem~\ref{existencein2d}, we get the following
\begin{theorem}
On every Riemannian torus,
minimal spines (exist and)
are embedded $\theta$-graphs
with geodesic arcs, forming angles of $\frac{2\pi}{3}$
at their two meeting points.
\end{theorem}

\subsection{Basics on flat tori.}
From now, we consider only constant curvature Riemannian metrics on $T$.
These correspond to flat (i.e. Euclidean) structures on $T$.
We make a very rough introduction about
the basic concepts of the topic to fix the notations,
referring the reader to \cite{Fm} for details.

Let us define
\begin{itemize}
\item the \emph{Teichm\"uller space} $\mathcal{T}$ 
of the torus as the space of isotopy classes
of unit-area flat structures
on $T$,
\item the \emph{mapping class group} $\MCG$ of the torus as
the group of isotopy classes of orientation 
preserving self-diffeomorphisms of $T$.
\item the \emph{moduli space} $\mathcal{M}$ of the torus   
 as the space of oriented isometry classes of unit-area flat structures
on $T$.
\end{itemize}

The group $\MCG$ acts on the set $\mathcal{T}$ and $\mathcal{M}$ is the quotient by this action. The set $\mathcal{T}$ is endowed with a natural metric (the Teichm\"uller metric), which makes it isometric to hyperbolic plane $\matH^2$. The action of the mapping class group can be identified with a discrete action by isometries. 

Visualizing the hyperbolic plane with the upper-half space model $\{\Im(z)>0\}\subset\mathbb{C}$, the action is given by integer M\"obius transformations and the mapping class group is isomorphic to the group $SL_2(\mathbb{Z})$ of unit-determinant $2\times 2$ integral matrices. The kernel of the action is $\{\pm I\}$, so that the moduli space $\mathcal{M}$ has the structure of a hyperbolic orbifold, with $\pi_1\mathcal{M}\simeq {\rm PSL}_2(\mathbb{Z})$.
A fundamental domain for the action is the hyperbolic semi-ideal triangle 
$$D=\big\{\vert z\vert\geq 1, \vert \Re (z)\vert \leq 1\big\},$$ 
with angles $\frac{\pi}{3}$, $\frac{\pi}{3}$ and $0$ (in grey in Figure \ref{flattori}).
The quotient $\mathcal M$ is the complete $(2,3,\infty)$-hyperbolic orbifold of finite area, with one cusp and two conical singularities of angles $\pi$ and $\frac{2\pi}{3}$. We see $\mathcal M$ as usual as $D$ with the boundary curves appropriately identified.

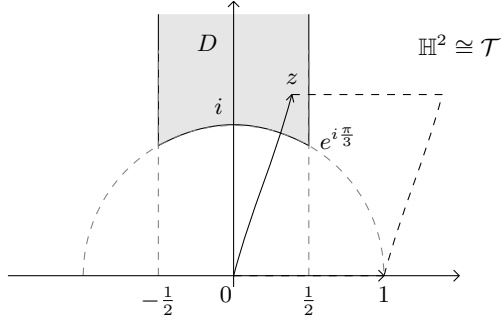
\begin{figure}
\begin{center}
\begin{tikzpicture}[scale=0.5]
\filldraw[fill=black!10!white]
(2,3.46) 
to[out= 90,in=-90, looseness=1] (2,7) -- (2,7)
to[out=-180, in=0, looseness=1] (-2, 7) --
(-2,7)to[out= -90,in=90, looseness=1] (-2,3.46) -- (-2,3.46) to[out=28.5, in=151.5, looseness=1]  (2,3.46);
\draw[color=white, very thick]
(-6,7) to[out= 0,in=180, looseness=1] (6,7);
\draw[dashed, color=black!50!white,scale=4,domain=-1.57: 1.57,
smooth,variable=\t,shift={(0,0)},rotate=0]plot({1*sin(\t r)},
{1*cos(\t r)}) ;
\draw[dashed, color=black!50!white]
(2,0) to[out=90, in=-90, looseness=1](2,6)
(-2,0) to[out=90, in=-90, looseness=1](-2,6);
\draw
(1.6,4.61) to[out= 105,in=-75, looseness=1] (1.55,4.82)
(1.4,4.67) to[out= 45,in=-135, looseness=1] (1.55,4.82);
\draw
(0,0) to[out= 75,in=-105, looseness=1] (1.55,4.82);
\draw[dashed] 
(0,0) to[out= 0,in=180, looseness=1] (4,0)
(1.55,4.82) to[out= 0,in=180, looseness=1](5.55,4.82)
(4,0) to[out= 75,in=-105, looseness=1](5.55,4.82);
\draw
(4,0) to[out= 135,in=-45, looseness=1] (3.85,0.15)
(4,0) to[out= -135,in=45, looseness=1] (3.85,-0.15);
\path[font=\scriptsize]
(0,4.5)node[left]{$i$}
(2,3.66)node[right]{$e^{i\frac{\pi}{3}}$}
(-0.2,0) node[below]{$0$}
(4,0) node[below]{$1$}
(2,0) node[below]{$\frac 12$}
(-2,0) node[below]{$-\frac {1}{2}$}
(1.55,4.82) node[above] {$z$}
(-0.7,6.7) node[below]{$D$}
(6,6.7) node[below]{$\mathbb{H}^2\cong \mathcal{T}$};
\draw
(6,0) to[out= 135,in=-45, looseness=1] (5.85,0.15)
(6,0) to[out= -135,in=45, looseness=1] (5.85,-0.15)
(0,7.3) to[out= -45,in=135, looseness=1] (0.15,7.15)
(0,7.3) to[out= 45,in=-135, looseness=1] (-0.15,7.15)
(0,-0.3) to[out= 90,in=-90, looseness=1] (0,7.3)
(-6,0) to[out= 0,in=-180, looseness=1] (6,0);
 \end{tikzpicture}
\end{center}
\caption{Every $z\in\mathbb H^2$ in upper half-plane represents a flat torus obtained by identifying the opposite edges of the parallelogram with vertices 0, 1, $z$, $z+1$. A fundamental domain $D$ for the action of ${\rm PSL}_2(\mathbb Z)$ is colored in grey.}
\label{flattori}
\end{figure}

We will also be interested in the following space:
\begin{itemize}
\item the \emph{non-oriented moduli space} $\Mnon$ is the space of all isometry classes of (unoriented) unit-area flat structures on $T$.
\end{itemize}

We have $\Mnon = \mathcal{M}/_\iota$ where $\iota$ is the isometric involution that sends an oriented flat torus to the same torus with opposite orientation: the lift of $\iota$ to the fundamental domain $D$ is the reflection with respect to the geodesic line $i\mathbb{R}^+$ and $\Mnon$ is the hyperbolic triangle orbifold 
$$\Mnon = D\cap\{\Re(z)\geq 0\}$$ 
with angles $\frac{\pi}{2}$, $\frac{\pi}{3}$ and $0$.
We do not need to fix an orientation on $T$ to talk about spines, hence we will work mainly with $\Mnon$.

Every flat torus $T$ has a continuum of isometries: the \emph{translations} by any vector in $\matR^2$ and the \emph{reflections} with respect to any point $x\in T$. The tori lying in the mirror sides of $\Mnon$ have special names and enjoy some additional isometries:
\begin{itemize}
\item the \emph{rectangular tori} are those lying in $i\mathbb{R}^+$,
\item the \emph{rhombic tori} are those in the other sides of $\Mnon$, namely $\{\vert z\vert=1\} \cup \{\Re(z)=\frac12 \}$.
\end{itemize}
These flat tori are obtained by identifying the opposite sides of a rectangle and a rhombus, respectively. The tori in the cone points $z=i$ and $e^{\frac{\pi i}{3}}$ are the \emph{square torus} and the \emph{hexagonal torus}. On the hexagonal torus, the length $d$ of the shortest diagonal of the rhombus equals the length $l$ of any of its sides, while we have $d\geq l$ and $d\leq l$ on the sides $|z|=1$ and $\Re = \frac 12$ respectively (we can call these rhombi \emph{fat} and \emph{thin}, repsectively).

The rectangular and rhombic tori are precisely the flat tori that admit orientation-reversing isometries.

Teichm\"uller and moduli spaces have natural Thurston and Mumford-Deligne compactifications. In the torus case, these are obtained respectively by adding the circle ``at infinity'' $\partial \mathbb H^2 = \mathbb{R}\cup\{\infty\}$ to $\mathcal T = \matH^2$ and a single point to $\mathcal M$ or $\Mnon$. We denote the latter compactifications by $\overline \calM$ and $\overline \Mnon$.

\subsection{Hexagons.}\label{minspi}
The study of minimal spines on a flat torus is intimately related to that of a particular class of Euclidean hexagons.

\begin{definition}
A \emph{semi-regular hexagon} is a Euclidean hexagon with all internal angles $\frac{2\pi}3$ and with congruent opposite sides. 
\end{definition}

We define the \emph{moduli space} $\calH$ as the space of all oriented semi-regular hexagons considered up to homotheties and orientation-preserving isometries. Similarly $\Hnon$ is defined by considering non-oriented hexagons and by quotienting by homotheties and all isometries. We get a map $\calH \to \Hnon$ that is at most 2-to-1.

Two opposite sides of a semi-regular hexagon are parallel and congruent. A semi-regular hexagon is determined up to isometry by the lengths $a,b,c >0$ of three successive sides, hence we get
\begin{align*}
\calH & = \big\{(a,b,c)\,|\,a,b,c>0\big\}/_{\R_{>0} \times A_3} \\
\Hnon & = \big\{(a,b,c)\,|\,a,b,c>0\big\}/_{\R_{>0} \times S_3}
\end{align*}
where the multiplicative group of positive real numbers $\matR_{>0}$ acts on the triples by rescaling, while $A_3$ and $S_3$ act by permuting the components. 

We can visualize the space $\Hnon$ in the positive orthant of $\matR^3$ by normalizing $(a,b,c)$ such that $a+b+c=1$ and $a\geq b\geq c$, and in this way $\Hnon$ is a triangle with one side removed (corresponding to $c=0$) as in Fig.~\ref{hexa:fig}. The other two sides parametrize the hexagons with $a=b\geq c$ and $a\geq b=c$. The regular hexagon is of course $(\frac 13, \frac 13, \frac 13)$. The space $\Hnon$ is naturally an orbifold with mirror boundary made by two half-lines and one corner reflector of angle $\frac \pi 3$.

The space $\H$ of oriented semi-regular hexagons is obtained by doubling the triangle $\Hnon$ along its two edges. Therefore $\H$ is topologically an open disc, and can be seen as an orbifold where the regular hexagon is a cone point of angle $\frac {2\pi}3$. The map $\H \to \Hnon$ may be interpreted as an orbifold cover of degree two.

The orbifold universal covering $\tilde{\mathcal{H}}\to\mathcal{H}$ is homeomorphic to the map $z\mapsto z^3$ from the complex plane to itself. The orbifold fundamental group $\pi_1\mathcal{H}$ is isomorphic to the cyclic group $\mathbb{Z}/3$, acting on the plane by rotations of angle $\frac{2\pi}{3}$.

\begin{center}
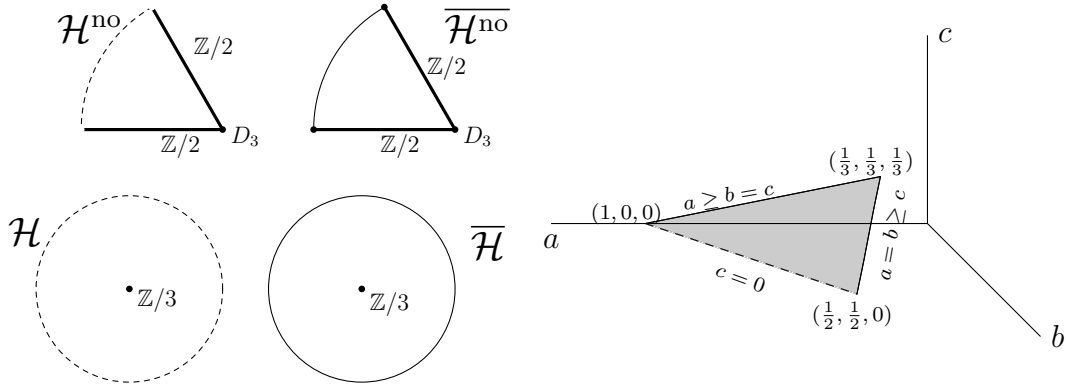
\begin{figure}[h]

\definecolor{ffffff}{rgb}{1,1,1}
\begin{tikzpicture}[line cap=round,line join=round,>=triangle 45,x=1.0cm,y=1.0cm,scale=0.5]

\filldraw[fill=black!20!white,scale=2.5,shift={(9,-1)}]
(-0.75,-0.75)to[out=161.57,in=-18.43,looseness=1](-3,0)--
(-3,0)to[out=11.3,in=-168.69,looseness=1](-0.5,0.5)--
(-0.5,0.5)to[out=-102,in=78,looseness=1](-0.75,-0.75);
\draw[scale=2.5,shift={(9,-1)}]
(0,0) to[out= 90,in=-90, looseness=1] (0,2) node[right]{$c$}
(0,0)to[out=180,in=0, looseness=1](-4,0) node[below]{$a$}
(0,0)to[out=-45,in=135,looseness=1](1.2,-1.2) node[right]{$b$} ;
\draw[color=white,scale=2.5,shift={(9,-1)}]
(-3,0)to[out=-18.43,in=161.57,looseness=1](-0.75,-0.75);
\draw[scale=2.5,shift={(9,-1)}]
(-3,0)to[out=11.3,in=-168.69,looseness=1](-0.5,0.5)
(-0.5,0.5)to[out=-102,in=78,looseness=1](-0.75,-0.75);
\draw[dashed,scale=2.5,shift={(9,-1)}]
(-3,0)to[out=-18.43,in=161.57,looseness=1](-0.75,-0.75);

\path[font=\tiny]
(17.3,-2.3) node[above,rotate=10]{$a\geq b=c$}
(17.7,-3.5) node[below,rotate=-20]{$c=0$}
(21.3,-4.3) node[right,rotate=80]{$a=b\geq c$}
(14.5,-2.8) node[above]{$(1,0,0)$}
(21,-1.5) node[above]{$(\frac13,\frac13,\frac13)$}
(20.5,-4.3) node[below]{$(\frac12,\frac12,0)$};

\clip(-6.04,-7.7) rectangle (17.12,4.68);
\draw [line width=1.2pt] (0,0)-- (3.76,0);
\draw [line width=1.2pt] (3.76,0)-- (1.88,3.26);
\draw [shift={(3.76,0)},dash pattern=on 2pt off 2pt]  plot[domain=2.09:3.14,variable=\t]({1*3.76*cos(\t r)+0*3.76*sin(\t r)},{0*3.76*cos(\t r)+1*3.76*sin(\t r)});
\draw (3.76,0.3) node[scale=0.75,anchor=north west] {$ D_3 $};
\draw (2.78,2.7) node[scale=0.75,anchor=north west] {$ \mathbb{Z}/2 $};
\draw (1.9,0.1) node[scale=0.75,anchor=north west] {$ \mathbb{Z}/2 $};
\draw (-1,3.42) node[scale=1.25,anchor=north west] {$ \mathcal{H}^\mathrm{no} $};
\begin{scriptsize}
\fill [color=ffffff] (0,0) circle (2.5pt);
\fill [color=ffffff] (1.88,3.26) circle (2.5pt);
\fill [color=black] (3.76,0) circle (2.5pt);
\end{scriptsize}
\draw [line width=1.2pt] (9.94,0)-- (6.18,0);
\draw [line width=1.2pt] (9.94,0)-- (8.06,3.26);
\draw [shift={(9.94,0)}] plot[domain=2.09:3.14,variable=\t]({1*3.76*cos(\t r)+0*3.76*sin(\t r)},{0*3.76*cos(\t r)+1*3.76*sin(\t r)});
\draw (9.94,0.3) node[scale=0.75,anchor=north west] {$ D_3 $};
\draw (8.98,2.18) node[scale=0.75,anchor=north west] {$ \mathbb{Z}/2 $};
\draw (7.74,0.1) node[scale=0.75,anchor=north west] {$ \mathbb{Z}/2 $};
\draw (9.3,3.62) node[scale=1.25,anchor=north west] {$ \overline{\mathcal{H}^\mathrm{no}} $};
\begin{scriptsize}
\fill [color=black] (9.94,0) circle (2.5pt);
\fill [color=black] (6.18,0) circle (2.5pt);
\fill [color=black] (8.06,3.26) circle (2.5pt);
\end{scriptsize}
\draw [dash pattern=on 2pt off 2pt] (1.28,-4.24) circle (2.48cm);
\draw (1.28,-4) node[scale=0.75,anchor=north west] {$ \mathbb{Z}/3 $};
\draw (7.46,-4.1) node[scale=0.75,anchor=north west] {$ \mathbb{Z}/3 $};
\draw (-2.3,-1.98) node[scale=1.25,anchor=north west] {$ \mathcal{H} $};
\begin{scriptsize}
\fill [color=black] (1.28,-4.24) circle (2.5pt);
\end{scriptsize}
\draw(7.46,-4.24) circle (2.48cm);
\draw (10,-2.16) node[scale=1.25,anchor=north west] {$ \overline{\mathcal{H}} $};
\begin{scriptsize}
\fill [color=black] (7.46,-4.24) circle (2.5pt);
\end{scriptsize}
\end{tikzpicture}

\caption{On the left, the orbifold structures on the moduli spaces of semi-regular hexagons, and their compactifications.
On the right, a parametrization of $\Hnon$ by a triangle with one side removed in $\mathbb{R}^3$.}
\label{hexa:fig}
\end{figure}
\end{center}

Both $\calH$ and $\Hnon$ have natural compactifications, obtained by adding the side with $c=0$, which consists of points $(a,1-a,0)$ with $a \in [\frac 12, 1]$. Each such point corresponds to some ``degenerate'' hexagon: the points with $a<1$ may be interpreted as parallelograms with angles $\frac{\pi}3$ and $\frac{2\pi}3$, while $(1,0,0)$ should be interpreted as a segment  -- a \emph{doubly degenerate} hexagon. The underlying spaces of the resulting compactifications $\overline{\calH}$ and $\overline{\Hnon}$ are both homeomorphic to closed discs.

\subsection{Spines.}
We now introduce two more moduli spaces $\calS$ and $\Snon$ which will turn out to be isomorphic to $\calH$ and $\Hnon$.

Let the \emph{moduli space} $\calS$ be the set of all pairs $(T, \Gamma)$, where $T$ is a flat oriented torus and $\Gamma \subset T$ a minimal spine, considered up to orientation-preserving isometries (that is, $(T,\Gamma) = (T', \Gamma')$ if there is an orientation-preserving isometry $\psi\colon T \to T'$ such that $\psi(\Gamma) = \Gamma'$). We define analogously $\Snon$ as the set of all pairs $(T,\Gamma)$ where $T$ is unoriented, quotiented by all isometries. Again we get a degree two orbifold covering $\calS \to \Snon$.

The opposite edges of a semi-regular hexagon $H$ are congruent, and by identifying them we get a flat torus $T$. The boundary $\partial H$ of the hexagon transforms into a minimal spine $\Gamma \subset T$ in the gluing process. This simple operation define two maps
$$\calH \longrightarrow \calS, \qquad \Hnon \longrightarrow \Snon.$$

\begin{proposition}
Both maps are bijections.
\end{proposition}
\begin{proof}
The inverse map is the following: given $(T, \Gamma)$, we cut $T$ along $\Gamma$ and get the original semi-regular hexagon $H$.
\end{proof}

We will therefore henceforth identify these moduli spaces and use the symbols $\calH$ and $\Hnon$ to denote the moduli spaces of both semi-regular hexagons and pairs $(T,\Gamma)$. Of course our aim is to use the first (hexagons) to study the second (spines in flat tori).

\subsection{The forgetful maps.}
The main object of this chapter is the characterization of the forgetful maps 
$$p\colon \calS \longrightarrow \calM, \qquad p\colon \Snon \longrightarrow \Mnon$$
that send $(T, \Gamma)$ to $T$ forgetting the minimal spine $\Gamma$. The fiber $p^{-1}(T)$ over an oriented flat torus $T\in \calM$ can be interpreted as the set of all minimal spines in $T$, considered up to orientation-preserving isometries of $T$. Likewise the fiber over an unoriented torus $T\in \Mnon$ is the set of minimal spines in $T$, considered up to all isometries.

\begin{figure}[h]
\begin{center}
\begin{tikzpicture}[line cap=round,line join=round,>=triangle 45,x=1.0cm,y=1.0cm,scale=0.65]
\clip(-0.04,-0.36) rectangle (13.6,6.63);
\draw [line width=1pt] (0.18,1.14)-- (1.48,0);
\draw [line width=1pt] (1.48,0)-- (2.47,0.34);
\draw [line width=1pt] (0.18,1.14)-- (1.11,5.8);
\draw [line width=1pt] (1.11,5.8)-- (2.1,6.14);
\draw [line width=1pt] (2.1,6.14)-- (3.4,5);
\draw [line width=1pt] (3.4,5)-- (2.47,0.34);
\draw [dash pattern=on 1pt off 1pt] (6.18,5.8)-- (8.46,5);
\draw [line width=1pt] (5.25,1.14)-- (6.18,5.8);
\draw [line width=1pt] (5.73,3.55) -- (5.81,3.45);
\draw [line width=1pt] (5.73,3.55) -- (5.62,3.49);
\draw [dash pattern=on 1pt off 1pt] (7.16,6.14)-- (7.04,5.5);
\draw [line width=1pt] (10.95,5.8)-- (13.24,5);
\draw [line width=1pt] (12.17,5.37) -- (12.06,5.31);
\draw [line width=1pt] (12.17,5.37) -- (12.13,5.49);
\draw [line width=1pt] (9.9,0.5)-- (12.18,-0.3);
\draw [line width=1pt] (11.11,0.07) -- (11.01,0.01);
\draw [line width=1pt] (11.11,0.07) -- (11.07,0.19);
\draw [line width=1pt] (12.18,-0.3)-- (13.24,5);
\draw [line width=1pt] (12.74,2.5) -- (12.82,2.41);
\draw [line width=1pt] (12.74,2.5) -- (12.63,2.44);
\draw [line width=1pt] (12.71,2.35) -- (12.79,2.25);
\draw [line width=1pt] (12.71,2.35) -- (12.6,2.29);
\draw [line width=1pt] (9.9,0.5)-- (10.95,5.8);
\draw [line width=1pt] (10.45,3.31) -- (10.53,3.21);
\draw [line width=1pt] (10.45,3.31) -- (10.34,3.25);
\draw [line width=1pt] (10.42,3.15) -- (10.5,3.05);
\draw [line width=1pt] (10.42,3.15) -- (10.31,3.09);
\draw [line width=1pt] (7.54,0.34)-- (8.46,5);
\draw [line width=1pt] (8.02,2.74) -- (8.1,2.65);
\draw [line width=1pt] (8.02,2.74) -- (7.91,2.69);
\draw [line width=1pt] (6.18,5.8)-- (7.16,6.14);
\draw [line width=1pt] (6.82,6.02) -- (6.78,5.9);
\draw [line width=1pt] (6.82,6.02) -- (6.71,6.09);
\draw [line width=1pt] (6.67,5.97) -- (6.63,5.85);
\draw [line width=1pt] (6.67,5.97) -- (6.56,6.04);
\draw [line width=1pt] (6.55,0)-- (7.54,0.34);
\draw [line width=1pt] (7.19,0.22) -- (7.15,0.1);
\draw [line width=1pt] (7.19,0.22) -- (7.09,0.29);
\draw [line width=1pt] (7.04,0.17) -- (7,0.05);
\draw [line width=1pt] (7.04,0.17) -- (6.94,0.23);
\draw [line width=1pt] (6.55,0)-- (5.25,1.14);
\draw [line width=1pt] (5.84,0.62) -- (5.96,0.64);
\draw [line width=1pt] (5.84,0.62) -- (5.83,0.5);
\draw [line width=1pt] (5.96,0.52) -- (6.08,0.54);
\draw [line width=1pt] (5.96,0.52) -- (5.95,0.39);
\draw [line width=1pt] (5.72,0.73) -- (5.84,0.75);
\draw [line width=1pt] (5.72,0.73) -- (5.71,0.6);
\draw [line width=1pt] (8.46,5)-- (7.16,6.14);
\draw [line width=1pt] (7.75,5.62) -- (7.88,5.64);
\draw [line width=1pt] (7.75,5.62) -- (7.75,5.49);
\draw [line width=1pt] (7.87,5.51) -- (8,5.54);
\draw [line width=1pt] (7.87,5.51) -- (7.87,5.39);
\draw [line width=1pt] (7.63,5.72) -- (7.76,5.75);
\draw [line width=1pt] (7.63,5.72) -- (7.63,5.6);
\draw (0.18,1.14)-- (1.11,5.8);
\draw (1.11,5.8)-- (2.1,6.14);
\draw (2.1,6.14)-- (3.4,5);
\draw (3.4,5)-- (2.47,0.34);
\draw (2.47,0.34)-- (1.48,0);
\draw (1.48,0)-- (0.18,1.14);
\draw [->,line width=2pt] (3.73,3.25) -- (5.48,3.25);
\draw [->,line width=2pt] (8.53,3.25) -- (10.05,3.25);
\path[font=\normalsize]
(0.59,3.75) node[anchor=north west] {$a$}
(0.52,0.63) node[anchor=north west] {$b$}
(2.05,0.24) node[anchor=north west] {$c$}
(2.93,2.81) node[anchor=north west] {$a$}
(2.8,5.94) node[anchor=north west] {$b$}
(1.51,6.08) node[anchor=north west] {$c$};
\end{tikzpicture}
\end{center}
\caption{The composition $p\colon \mathcal{H}\to\mathcal{S}\to\mathcal{M}$.}
\label{compositions}
\end{figure}
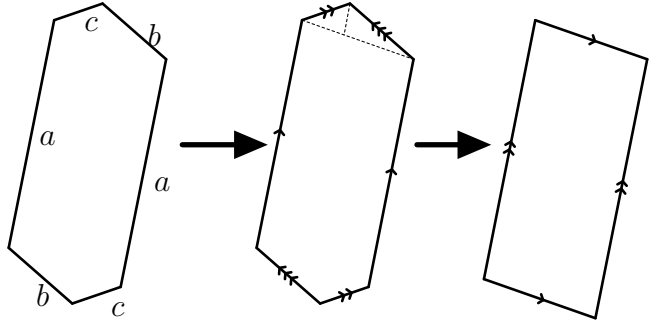

As we said above, we identify $\calH$, $\Hnon$ with $\calS$, $\Snon$ and consider the compositions (which we still name by $p$)
$$p\colon\calH \stackrel\sim\longrightarrow \calS \longrightarrow \calM, \qquad p\colon\Hnon \stackrel\sim\longrightarrow \Snon \longrightarrow \Mnon.$$
The map $p$ is described geometrically in Fig.~\ref{compositions}. 

\begin{figure}
\begin{center}
\begin{minipage}[c]{.40\textwidth}
\centering
\begin{tikzpicture}[scale=2,line cap=round,line join=round,>=triangle 45,x=1.0cm,y=1.0cm,scale=0.75]
\path[font=\footnotesize]
(-0.42,2.78) node[above]{$z=\tilde{p}(a,b,c)$};
\path[font=\footnotesize]
(-0.07,-0.01) node[below] {$0$}
(1.01,-0.01) node[below] {$1$}
(0.1,0.4) node[below]{$c'$}
(0.6,0.4) node[below]{$b'$}
(-0.15, 2.3) node[above]{$a'$};
\draw[->,color=black] (-0.8,0) -- (2.3,0);
\draw[->,color=black] (0,-0.27) -- (0,3.3);

\clip(-1.89,-0.27) rectangle (4.18,2.98);
\fill[fill=black,fill opacity=0.1] (0,0) -- (-0.42,2.88) -- (0.58,2.88) -- (1,0) -- cycle;

\draw (0,0)-- (-0.42,2.88);
\draw (-0.42,2.88)-- (0.58,2.88);
\draw (0.58,2.88)-- (1,0);
\draw (1,0)-- (0,0);

\draw [line width=1pt] (0.23,0.22)-- (1,0);
\draw [line width=1pt] (0,0)-- (0.23,0.22);
\draw [line width=1pt] (0.23,0.22)-- (-0.42,2.88);
\end{tikzpicture}
\end{minipage}
\hspace{10mm}
\begin{minipage}[c]{.50\textwidth}
\centering
\caption{
How to construct $z$ from $(a', b', c')$: we pick the tripod with angles $\frac{2\pi}3$ and lengths $a', b', c'$, and place it in the upper half-plane so that the endpoints of  the edges $c'$ and $b'$ lie in $0$ and $1$. The point $z$ is the endpoint of the edge $a'$.}
\label{ptilde:fig}
\end{minipage}
\end{center}
\end{figure}
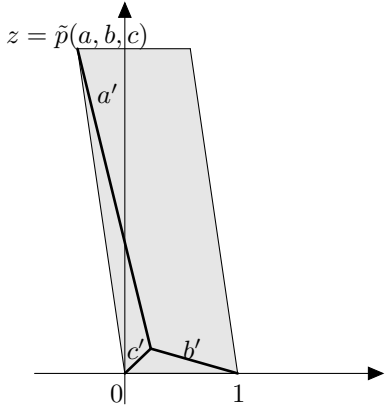

\begin{proposition}
The map 
\begin{align*}
\tilde p\colon \Hnon & \longrightarrow \matH^2 \\
(a,b,c) & \longmapsto \frac{2c^2-ab+ac+bc}{2(b^2+c^2+bc)}+i\frac{\sqrt{3}}{2}\,\frac{ab+bc+ac}{b^2+c^2+bc}
\end{align*}
is a lift of $p\colon \Hnon \to \Mnon$. 
The map $\tilde p$ is injective and sends homeomorphically the triangle $\Hnon$ onto the light grey domain drawn in Fig.~\ref{ptilde}.
\end{proposition}
\begin{proof}
The map $\tilde p$ is defined in Fig.~\ref{ptilde:fig}. 
We first rescale the triple $(a,b,c)$, which by hypothesis satisfies $a\geq b \geq c$ and $a+b+c=1$, to $(a',b',c')$, 
by multiplying each term by $1/\sqrt{b^2+c^2+bc}$. 
Now we pick the \emph{tripod} with one vertex and three edges of length $a', b', c'$ and with angles $\frac{2\pi}3$,
and we place it in the half-space with two vertices in $0$ and $1$ as shown in the figure (we can do this thanks to the rescaling). 
The third vertex goes to some $z\in \matH^2$ and we set $\tilde p(a,b,c) = z$.

The map $\tilde p$ is clearly a lift of $p$ and it only remains to determine an explicit expression for $\tilde p$. 
Applying repeatedly the Carnot Theorem, we get:
\begin{align*}
{\rm Arg}(z)&=\cos^{-1} \frac{2c^2-ab+ac+bc}{2\sqrt{\left(b^2+c^2+bc\right)\left(a^2+c^2+ac\right)}}\,,\\
\vert z\vert&=\sqrt{\frac{a^2+c^2+ac}{b^2+c^2+bc}}\,.
\end{align*}
Finally,
$$
\tilde p(a,b,c) =z=\frac{2c^2-ab+ac+bc}{2(b^2+c^2+bc)}+i\frac{\sqrt{3}}{2}\,\frac{ab+bc+ac}{b^2+c^2+bc}.
$$
The proof is complete.
\end{proof}

\begin{figure}[h]
\begin{center}
\begin{tikzpicture}[scale=1.65,line cap=round,line join=round,>=triangle 45,x=1.0cm,y=1.0cm]
\filldraw[fill=black!30!white,scale=0.5]
(1,7.3) 
to[out= -90,in=90, looseness=1] (1,1.73) -- (1,1.73)
to[out=29, in=151, looseness=1] (3, 1.73) --
(3,1.73)to[out= 60,in=-120, looseness=1] (6.22,7.3) -- (6.22,7.3) to[out=180, in=0, looseness=1]  (1,7.3);
\filldraw[fill=black!5!white,scale=0.5]
(-4.2,7.3) 
to[out= -60,in=120, looseness=1] (-1,1.73) -- (-1,1.73)
to[out=29, in=151, looseness=1] (1, 1.73) --
(1,1.73)to[out= 90,in=-90, looseness=1] (1,7.3) -- (1,7.3) to[out=180, in=0, looseness=1]  (-4.2,7.3);
\draw[color=white,scale=0.5,very thick]
(-4.4,7.3) 
to[out= 0,in=180, looseness=1] (6.4,7.3) ;
\draw[->,color=black] (-2.84,0) -- (4.51,0);
\draw[->,color=black] (0,-0.19) -- (0,3.83);
\draw [shift={(0,0)},line width=0.5pt,dotted]  
plot[domain=0:3.14,variable=\t]({1*1*cos(\t r)+0*1*sin(\t r)},{0*1*cos(\t r)+1*1*sin(\t r)});
\draw [shift={(1,0)},line width=0.5pt,dotted]  
plot[domain=0:3.14,variable=\t]({1*1*cos(\t r)+0*1*sin(\t r)},{0*1*cos(\t r)+1*1*sin(\t r)});
\clip(-2.84,-0.19) rectangle (4.51,3.73);
\draw [shift={(0,0)},line width=0.7pt]  plot[domain=1.05:2.09,variable=\t]({1*1*cos(\t r)+0*1*sin(\t r)},{0*1*cos(\t r)+1*1*sin(\t r)});
\draw [shift={(1,0)},line width=0.7pt]  plot[domain=1.05:2.09,variable=\t]({1*1*cos(\t r)+0*1*sin(\t r)},{0*1*cos(\t r)+1*1*sin(\t r)});
\path[font=\tiny]
(0.95,0.03) node[anchor=north west] {$1$};
\path[font=\normalsize]
(-1.4,2.81) node[anchor=north west] {$ \tilde p (\Hnon)$}
(2.2,2) node[anchor=north west] {$ \tilde p(\calH) $};
\draw [line width=0.7pt,color=white] (0.5,3.66) -- (0.5,-0.19);
\draw [line width=0.7pt,color=white] (-0.5,0.87)-- (-2.11,3.66);
\draw [line width=1pt,color=white] (3.11,3.66)-- (1.5,0.87);
\draw [line width=0.5pt,dotted] (0.5,3.66) -- (0.5,-0.19);
\draw [line width=0.5pt,dashed] (-0.5,0.87)-- (-2.11,3.66);
\draw [dotted] (-1.5,3.66)-- (-1.5,0);
\draw [dotted] (-0.5,3.66)-- (-0.5,0);
\draw [dotted] (1.5,3.66)-- (1.5,0);
\draw [dotted] (2.5,3.66)-- (2.5,0);
\draw [line width=1pt,dashed] (3.11,3.66)-- (1.5,0.87);
\draw [line width=1pt] (0.5,3.66)-- (0.5,0.87);
\path[font=\tiny]
(0.5, 2.2) node[above,rotate=90]{$a\geq b=c$}
(-1.2, 2.2) node[below,rotate=-60]{$c=0$}
(-0.48, 0.8) node[right]{$a=b\geq c$};
\end{tikzpicture}
\caption{The lift $\tilde p$ sends $\Hnon$ homeomorphically to the light grey domain on the left.
It sends the two sides $a\geq b = c$ and $a=b\geq c$ of $\Hno$ to two geodesic lines in $\matH^2$, and sends the line at infinity $c=0$ to the constant-curvature line dashed in the figure. In the oriented setting, the lift $\tilde p$ sends $\calH$ to the whole bigger grey domain but is discontinuous on the segment $s$ of hexagons of type $a=b\geq c$.}
\label{ptilde}
\end{center}
\end{figure}

Fig.~\ref{ptilde} shows that the lift $\tilde p$ sends the two sides $a=b\geq c$ and $a\geq b=c$ of the triangle $\Hnon$ to geodesic arcs in $\matH^2$. Recall that the compactification $\overline{\Hnon}$ is obtained by adding the singly-degenerate parallelograms $(a,1-a,0)$ with $a\in [\frac 12, 1)$ and the doubly-degenerate $P=(1,0,0)$. The map $\tilde p$ also extends to the parallelograms, and sends them to the dashed line in Fig.~\ref{ptilde:fig}, but it does \emph{not} extend continuously to $P$, not even as a map from $\overline{\Hnon}$ to $\overline{\matH^2}$. However, the map $p$ from moduli spaces does extend.

\begin{proposition}
The map $p\colon \Hnon \to \Mnon$ extends continuously to a map $p\colon \overline{\Hnon} \to \overline{\Mnon}$.
\end{proposition}
\begin{proof}
Send the doubly degenerate point $P$ to the point at infinity in $\overline{\calM}$.
\end{proof}

The oriented picture is easily deduced from the non-oriented one. We can lift the map $p\colon \calH \to \calM$ to a map $\tilde p\colon \calH \to \matH^2$ in Teichm\"uller space whose image is the bigger (both light and dark) grey domain in Fig.~\ref{ptilde}, however the map $\tilde p$ is discontinuous at the segment consisting of all hexagons with $a=b\geq c$, which is sent to one of the two curved geodesic arcs in the picture. The map $p\colon \calH \to \calM$ extends continuously to a map $p\colon \overline{\calH} \to \overline{\calM}$.

The M\"obius transformation $z\mapsto\frac{2iz+\sqrt{3}-i}{2z-1+\sqrt{3}i}$ gives a more symmetric picture of the image $\tilde{p}(\mathcal{H})$ inside the hyperbolic plane in the Poincar\'e disc model, as can be seen in Figure~\ref{disc}.

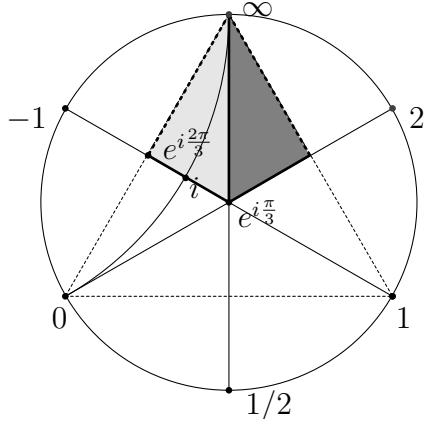
\begin{figure}[h]
\begin{center}
\definecolor{uququq}{rgb}{0.25,0.25,0.25}
\begin{tikzpicture}[line cap=round,line join=round,>=triangle
45,x=1.0cm,y=1.0cm,scale=2.5]
\clip(-2,-1.15) rectangle (2.76,1.4);
\fill[line width=2pt,fill=black,fill opacity=0.1] (0,1) -- (-0.43,0.25) --
(0,0) -- cycle;
\fill[line width=2pt,fill=black,fill opacity=0.5] (0,1) -- (0,0) --
(0.43,0.25) -- cycle;
\draw(0,0) circle (1cm);
\draw [line width=1pt] (0,1)-- (0,0);
\draw (0,0)-- (-0.87,0.5);
\draw (-0.87,-0.5)-- (0,0);
\draw (0,0)-- (0,-1);
\draw (0,0)-- (0.87,-0.5);
\draw (0,0)-- (0.87,0.5);
\draw [shift={(-1.73,1)}] plot[domain=-1.05:0,variable=\t]({1*1.73*cos(\t
r)+0*1.73*sin(\t r)},{0*1.73*cos(\t r)+1*1.73*sin(\t r)});
\draw [dash pattern=on 1pt off 1pt] (0,1)-- (-0.87,-0.5);
\draw [dash pattern=on 1pt off 1pt] (0,1)-- (0.87,-0.5);
\draw (-0.27,0.19) node[anchor=north west] {$i$};
\draw (-0.41,0.45) node[anchor=north west] {$e^{i\frac{2\pi}{3}}$};
\draw (-0.01,0.07) node[anchor=north west] {$e^{i\frac{\pi}{3}}$};
\draw (0.01,1.12) node[anchor=north west] {$\infty$};
\draw (-1.24,0.56) node[anchor=north west] {$-1$};
\draw (-1,-0.5) node[anchor=north west] {$0$};
\draw (0.83,-0.5) node[anchor=north west] {$1$};
\draw (0.9,0.56) node[anchor=north west] {$2$};
\draw [dash pattern=on 1pt off 1pt] (-0.87,-0.5)-- (0.87,-0.5);
\draw (0.03,-0.94) node[anchor=north west] {1/2};
\draw [line width=1pt] (-0.43,0.25)-- (0,0);
\draw [line width=1pt] (0,0)-- (0.43,0.25);
\draw [line width=1pt,dash pattern=on 1pt off 2pt] (0.43,0.25)-- (0,1);
\draw [line width=1pt,dash pattern=on 1pt off 2pt] (0,1)-- (-0.43,0.25);
\begin{scriptsize}
\fill [color=black] (0,0) circle (.5pt);
\fill [color=black] (0.87,-0.5) circle (.5pt);
\fill [color=black] (-0.87,0.5) circle (.5pt);
\fill [color=black] (0,-1) circle (.5pt);
\fill [color=uququq] (0,1) circle (.5pt);
\fill [color=uququq] (0.87,0.5) circle (.5pt);
\fill [color=black] (-0.87,-0.5) circle (.5pt);
\fill [color=black] (-0.23,0.13) circle (.5pt);
\fill [color=black] (-0.43,0.25) circle (.5pt);
\fill [color=black] (0.87,-0.5) circle (.5pt);
\end{scriptsize}
\end{tikzpicture}
\caption{A picture of $\tilde{p}(\H)$  in the Poincar\'e disc model. }
\label{disc}
\end{center}
\end{figure}

\begin{remark}\rm
Now, it is possible to quantify the ``distance'' between two spines on different flat tori, in a natural way. 
Indeed, thanks to the map $p$ and the Teichm\"uller metric, the moduli space of spines 
can be endowed with a  structure of hyperbolic orbifold. 

For, identify $-$ by a representation $\rho$ $-$ the  action of the orbifold fundamental group $\pi_1\H$ 
on the universal covering $\tilde{\H}$ with the action on the hyperbolic plane of the group $\Gamma$, 
generated by an order-three elliptic rotation about the point $e^{i\frac{\pi}{3}}$.
To get a developing map $d:\tilde{\mathcal{H}}\to\mathbb{H}^2$ for the hyperbolic structure on $\H$, 
lift $\rho$-equivariantly the map $p$.

The developed image $d(\tilde{\H})$ of $\H$ is the interior in $\mathbb{H}^2$ of the regular non-geodesic ideal triangle joining the points $\infty$, 0, 1 in Figure~\ref{disc}. The set $\tilde{p}(\H)$ is a fundamental domain for the action of $\Gamma$ and the space $\H$ is identified with the quotient $d(\tilde{\mathcal{H}})/\Gamma$ .

Hence, the moduli space of spines is an infinite-area incomplete hyperbolic orbifold. Its completion, supported on the complement of the doubly degenerate hexagon $\overline{\mathcal{H}}\setminus\{P\}$, is an orbifold with non-geodesic boundary. Its boundary is an infinite constant-curvature line whose points are (all the) four-valent geodesic spines with alternating angles $\frac{\pi}{3}$ and $\frac{2\pi}{3}$.

Similarly, $\overline{\H^\mathrm{no}}\setminus\{P\}$ is a complete infinite-area hyperbolic orbifold with (unbounded, non-geodesic) boundary and two mirror edges.
\end{remark}

\subsection{The number of minimal spines.}
We are ready to determine the fiber $p^{-1}(T)$ for every flat torus $T$ in $\calM$ and $\Mnon$, which may be interpreted as the set of all minimal spines in $T$, considered up to (orientation-preservingly or all) isometries of $T$. 

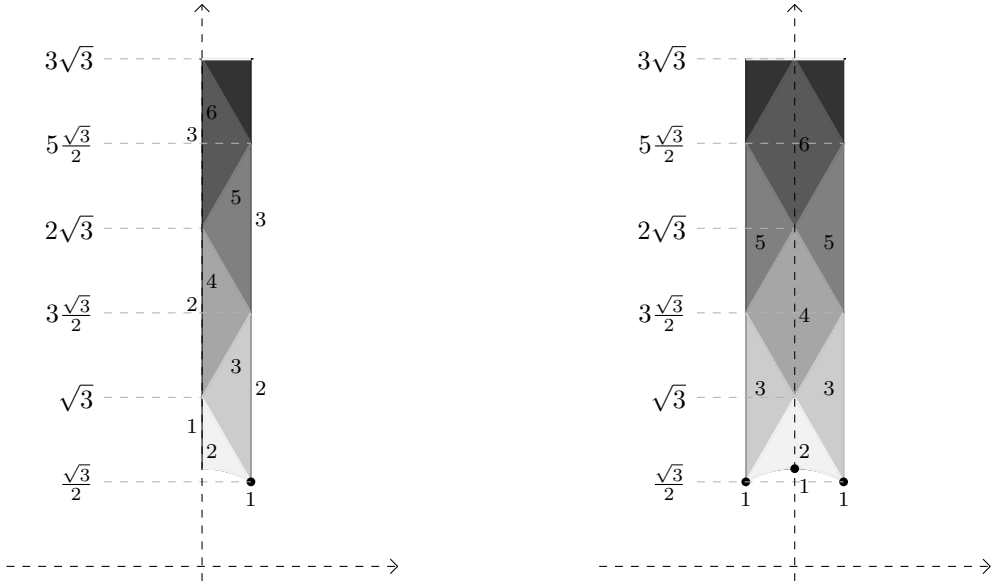
\begin{figure}[h!]
\begin{center}
\begin{minipage}[c]{.4\textwidth}
\centering
\begin{center}
\begin{tikzpicture}[scale=0.65]
\filldraw[fill=black!5!white]
(0,3.46) 
to[out= -90,in=90, looseness=1] (0,2) -- (0,2)
to[out=0, in=151, looseness=1] (1, 1.73) --
(1,1.73)to[out= 120,in=-60, looseness=1] (0,3.46); 
\filldraw[fill=black!20!white]
(0,3.46) 
to[out= -60,in=120, looseness=1] (1,1.73) -- (1,1.73)
to[out=90, in=-90, looseness=1] (1,5.19) --
(1,5.19) to[out= -120,in=60, looseness=1] (0,3.46); 
\filldraw[fill=black!35!white]
(0,6.92) 
to[out=-90,in=90, looseness=1] (0,3.46)--
(0,3.46)to[out= 60,in=-120, looseness=1] (1,5.19)--
(1,5.19)to[out= 120,in=-60, looseness=1] (0,6.92); 
\filldraw[fill=black!50!white,shift={(0,3.46)}]
(0,3.46) 
to[out= -60,in=120, looseness=1] (1,1.73) -- (1,1.73)
to[out=90, in=-90, looseness=1] (1,5.19) --
(1,5.19) to[out= -120,in=60, looseness=1] (0,3.46); 
\filldraw[fill=black!65!white]
(0,10.39) 
to[out=-90,in=90, looseness=1] (0,6.92)--
(0,6.92)to[out= 60,in=-120, looseness=1] (1,8.66)--
(1,8.66)to[out= 120,in=-60, looseness=1] (0,10.39);
\filldraw[fill=black!80!white]
(0,10.39) 
to[out= -60,in=120, looseness=1] (1,8.66) -- (1,8.66)
to[out=90, in=-90, looseness=1] (1,10.39) --
(1,10.39) to[out= 0,in=180, looseness=1] (0,10.39);  

\draw[scale=2,domain=0: 0.52,
smooth,variable=\t,shift={(0,0)},rotate=0]plot({1*sin(\t r)},
{1*cos(\t r)}) ;
\draw[color=black!5!white,thick]
(0,2)to[out=0, in=151, looseness=1] (1, 1.73)
(1,1.73)to[out=120, in=-60, looseness=1](0,3.46); 
\draw[color=black!20!white,thick]
(0,3.46)to[out=60, in=-120, looseness=1](1,5.19);
\draw[color=black!20!white]
(1,1.73) to[out=90, in=-90, looseness=1](1,5.19);
\draw[color=black!35!white,thick]
(1,5.19)to[out=120, in=-60, looseness=1](0,6.92);
\draw[color=black!50!white]
(1,5.19) to[out=90, in=-90, looseness=1](1,8.66);
\draw[color=black!50!white,thick]
(0,6.92)to[out=60, in=-120, looseness=1](1,8.66);
\draw[color=black!65!white,thick]
(1,8.66)to[out=120, in=-60, looseness=1](0,10.39); 
\draw[color=black!80!white]
(1,8.66) to[out=90, in=-90, looseness=1](1,10.39);
\draw[color=black!5!white, very thick]
(0,10.39) to[out=0, in=180, looseness=1](1,10.39);

\path
(1,1.73) node[circle,fill=black,scale=0.3] {};

\draw[color=black!30!white,dashed]
(-2,10.39)to[out=0, in=180, looseness=1](0,10.39)
(-2,8.66)to[out=0, in=180, looseness=1](1,8.66)
(-2,6.92)to[out=0, in=180, looseness=1](0,6.92)
(-2,5.19)to[out=0, in=180, looseness=1](1,5.19)
(-2,3.46)to[out=0, in=180, looseness=1](0,3.46)
(-2,1.73)to[out=0, in=180, looseness=1](1,1.73);
\path[font=\footnotesize]
(-2,10.39)node[left]{$3\sqrt{3}$}
(-2,8.66)node[left]{$5\frac{\sqrt{3}}{2}$}
(-2,6.92)node[left]{$2\sqrt{3}$}
(-2,5.19)node[left]{$3\frac{\sqrt{3}}{2}$}
(-2,3.46)node[left]{$\sqrt{3}$}
(-2,1.73)node[left]{$\frac{\sqrt{3}}{2}$};
\path[font=\tiny,color=black]
(1,1.73)node[below]{$1$}
(-0.2,3.23)node[below]{$1$}
(0.2,2.73)node[below]{$2$}
(-0.2,5.73)node[below]{$2$}
(1.2,4)node[below]{$2$}
(0.7,4.46)node[below]{$3$}
(-0.2,9.2)node[below]{$3$}
(1.2,7.46)node[below]{$3$}
(0.2,6.19)node[below]{$4$}
(0.7,7.92)node[below]{$5$}
(0.2,9.66)node[below]{$6$};
\draw
(4,0) to[out= 135,in=-45, looseness=1] (3.85,0.15)
(4,0) to[out= -135,in=45, looseness=1] (3.85,-0.15)
(0,11.5) to[out= -45,in=135, looseness=1] (0.15,11.35)
(0,11.5) to[out= 45,in=-135, looseness=1] (-0.15,11.35);
\draw[dashed]
(0,-0.3) to[out= 90,in=-90, looseness=1] (0,11.5)
(-4,0) to[out= 0,in=-180, looseness=1] (4,0);
\end{tikzpicture}
\end{center}

\end{minipage}
\hspace{10mm}
\begin{minipage}[c]{.4\textwidth}
\centering

\begin{center}
\begin{tikzpicture}[scale=0.65]
\filldraw[fill=black!5!white]
(0,3.46) 
to[out= -120,in=60, looseness=1] (-1,1.73) -- (-1,1.73)
to[out=29, in=151, looseness=1] (1, 1.73) --
(1,1.73)to[out= 120,in=-60, looseness=1] (0,3.46); 
\filldraw[fill=black!20!white]
(0,3.46) 
to[out= -120,in=60, looseness=1] (-1,1.73) -- (-1,1.73)
to[out=90, in=-90, looseness=1] (-1,5.19) --
(-1,5.19) to[out= -60,in=120, looseness=1] (0,3.46); 
\filldraw[fill=black!20!white]
(0,3.46) 
to[out= -60,in=120, looseness=1] (1,1.73) -- (1,1.73)
to[out=90, in=-90, looseness=1] (1,5.19) --
(1,5.19) to[out= -120,in=60, looseness=1] (0,3.46); 
\filldraw[fill=black!35!white]
(0,6.92) 
to[out= -120,in=60, looseness=1] (-1,5.19) --
(-1,5.19)to[out= -60,in=120, looseness=1] (0,3.46)--
(0,3.46)to[out= 60,in=-120, looseness=1] (1,5.19)--
(1,5.19)to[out= 120,in=-60, looseness=1] (0,6.92); 
\filldraw[fill=black!50!white,shift={(0,3.46)}]
(0,3.46) 
to[out= -120,in=60, looseness=1] (-1,1.73) -- (-1,1.73)
to[out=90, in=-90, looseness=1] (-1,5.19) --
(-1,5.19) to[out= -60,in=120, looseness=1] (0,3.46); 
\filldraw[fill=black!50!white,shift={(0,3.46)}]
(0,3.46) 
to[out= -60,in=120, looseness=1] (1,1.73) -- (1,1.73)
to[out=90, in=-90, looseness=1] (1,5.19) --
(1,5.19) to[out= -120,in=60, looseness=1] (0,3.46); 
\filldraw[fill=black!65!white]
(0,10.39) 
to[out= -120,in=60, looseness=1] (-1,8.66) --
(-1,8.66)to[out= -60,in=120, looseness=1] (0,6.92)--
(0,6.92)to[out= 60,in=-120, looseness=1] (1,8.66)--
(1,8.66)to[out= 120,in=-60, looseness=1] (0,10.39);
\filldraw[fill=black!80!white]
(0,10.39) 
to[out= -120,in=60, looseness=1] (-1,8.66) -- (-1,8.66)
to[out=90, in=-90, looseness=1] (-1,10.39) --
(-1,10.39) to[out= 0,in=180, looseness=1] (0,10.39); 
\filldraw[fill=black!80!white]
(0,10.39) 
to[out= -60,in=120, looseness=1] (1,8.66) -- (1,8.66)
to[out=90, in=-90, looseness=1] (1,10.39) --
(1,10.39) to[out= 0,in=180, looseness=1] (0,10.39);  
\draw[scale=2,domain=-0.52: 0.52,
smooth,variable=\t,shift={(0,0)},rotate=0]plot({1*sin(\t r)},
{1*cos(\t r)}) ;
\draw[color=black!5!white,thick]
(-1,1.73)to[out=29, in=151, looseness=1] (1, 1.73)
(-1,1.73)to[out=60, in=-120, looseness=1](0,3.46)
(1,1.73)to[out=120, in=-60, looseness=1](0,3.46); 
\draw[color=black!20!white,thick]
(0,3.46)to[out=60, in=-120, looseness=1](1,5.19)
(0,3.46)to[out=120, in=-60, looseness=1](-1,5.19); 
\draw[color=black!20!white]
(1,1.73) to[out=90, in=-90, looseness=1](1,5.19)
(-1,1.73) to[out=90, in=-90, looseness=1](-1,5.19);
\draw[color=black!35!white,thick]
(-1,5.19)to[out=60, in=-120, looseness=1](0,6.92)
(1,5.19)to[out=120, in=-60, looseness=1](0,6.92);
\draw[color=black!50!white]
(1,5.19) to[out=90, in=-90, looseness=1](1,8.66)
(-1,5.19) to[out=90, in=-90, looseness=1](-1,8.66);
\draw[color=black!50!white,thick]
(0,6.92)to[out=60, in=-120, looseness=1](1,8.66)
(0,6.92)to[out=120, in=-60, looseness=1](-1,8.66);
\draw[color=black!65!white,thick]
(-1,8.66)to[out=60, in=-120, looseness=1](0,10.39)
(1,8.66)to[out=120, in=-60, looseness=1](0,10.39); 
\draw[color=black!80!white]
(1,8.66) to[out=90, in=-90, looseness=1](1,10.39)
(-1,8.66) to[out=90, in=-90, looseness=1](-1,10.39);
\draw[color=black!5!white, very thick]
(-1,10.39) to[out=0, in=180, looseness=1](1,10.39);

\path
(-1,1.73) node[circle,fill=black,scale=0.3] {}
(0,2) node[circle,fill=black,scale=0.3] {}
(1,1.73) node[circle,fill=black,scale=0.3] {};

\draw[color=black!30!white,dashed]
(-2,10.39)to[out=0, in=180, looseness=1](0,10.39)
(-2,8.66)to[out=0, in=180, looseness=1](1,8.66)
(-2,6.92)to[out=0, in=180, looseness=1](0,6.92)
(-2,5.19)to[out=0, in=180, looseness=1](1,5.19)
(-2,3.46)to[out=0, in=180, looseness=1](0,3.46)
(-2,1.73)to[out=0, in=180, looseness=1](1,1.73);
\path[font=\footnotesize]
(-2,10.39)node[left]{$3\sqrt{3}$}
(-2,8.66)node[left]{$5\frac{\sqrt{3}}{2}$}
(-2,6.92)node[left]{$2\sqrt{3}$}
(-2,5.19)node[left]{$3\frac{\sqrt{3}}{2}$}
(-2,3.46)node[left]{$\sqrt{3}$}
(-2,1.73)node[left]{$\frac{\sqrt{3}}{2}$};
\path[font=\tiny,color=black]
(1,1.73)node[below]{$1$}
(0.2,2)node[below]{$1$}
(-1,1.73)node[below]{$1$}
(0.2,2.73)node[below]{$2$}
(-0.7,4)node[below]{$3$}
(0.7,4)node[below]{$3$}
(0.2,5.5)node[below]{$4$}
(-0.7,7)node[below]{$5$}
(0.7,7)node[below]{$5$}
(0.2,9)node[below]{$6$};
\draw
(4,0) to[out= 135,in=-45, looseness=1] (3.85,0.15)
(4,0) to[out= -135,in=45, looseness=1] (3.85,-0.15)
(0,11.5) to[out= -45,in=135, looseness=1] (0.15,11.35)
(0,11.5) to[out= 45,in=-135, looseness=1] (-0.15,11.35);
\draw[dashed]
(0,-0.3) to[out= 90,in=-90, looseness=1] (0,11.5)
(-4,0) to[out= 0,in=-180, looseness=1] (4,0);
\end{tikzpicture}
\end{center}

\end{minipage}

\end{center}
\caption{The number of minimal spines on each flat torus, in the unoriented and oriented setting. At each point $z$ in $\Mno$ or $\calM$, the number $c^{\rm no}(z)$ or $c(z)$ is the smallest number among all numbers written on the adjacent strata (both functions are lower semi-continuous).}
\label{c:fig}
\end{figure}

\begin{theorem}\label{fibra}
Every unoriented flat torus $T$ has a finite number $c^{\rm no}(T)$ of minimal spines up to isometries of $T$. Analogously, every oriented flat torus $T$ has a finite number $c(T)$ of minimal spines up to orientation-preserving isometries. The proper functions $c^{\rm no} \colon \Mnon \to \matN$ and $c\colon \calM \to \matN$ are shown in Fig.~\ref{c:fig}.
\end{theorem}

\begin{proof}
We can identify $\Hnon$ with its image $H=\tilde p(\Hnon)$ and note that the restriction of the projection $\pi\colon \matH^2 \to \Mno$ to $H$ is finite-to-one. Indeed, for $z,z'\in H$, $\pi(z)=\pi(z')$ if and only if $\Im z' = \Im z$ and $\Re z' = \pm \Re z + n$ for some integer $n$. The number $c^{\rm no}(z)$ is the cardinality of the fiber $p^{-1}(z)$ and is easily shown to be as in Fig.~\ref{c:fig}-(left).

The oriented case is treated analogously: in that case $\tilde p(\calH)$ is the bigger grey zone in $\matH^2$ and $p(z)=p(z')$ if and only if $z'=z+n$ for some $n\in\matZ$. We get the picture in Fig.~\ref{c:fig}-(right).
\end{proof}

\begin{figure}
 \begin{center}
  \includegraphics[width = 14 cm]{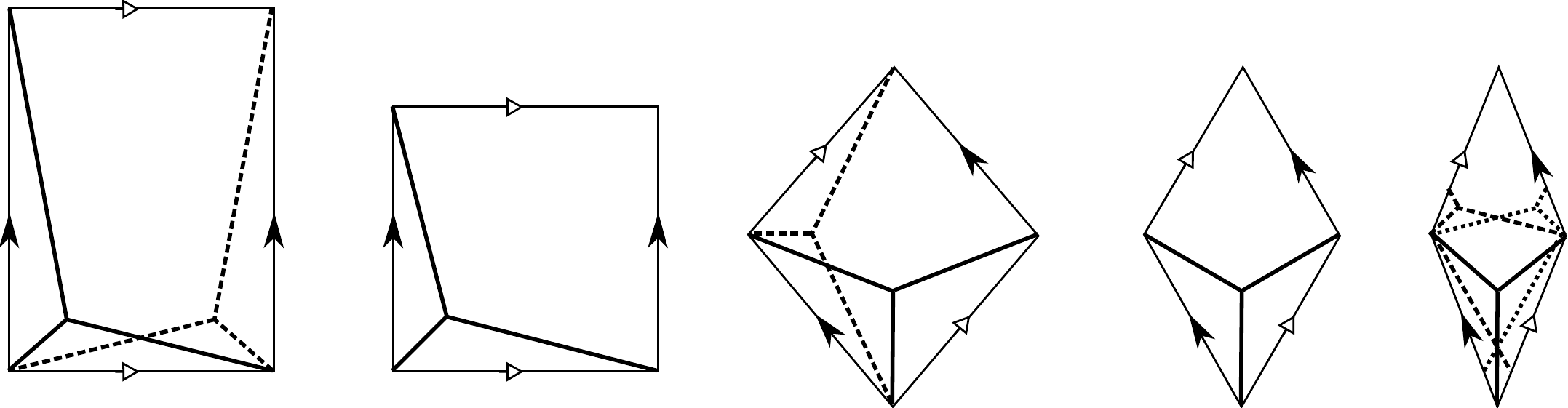}
 \end{center}
 \caption{The minimal spines in some oriented tori having additional symmetries. As we pass from rectangular to thin rhombic, we find 2, 1, 2, 1, and then 3 spines up to orientation-preserving isometries. They reduce to 1, 1, 2, 1, 2 up to all isometries, including orientation-reversing ones. Rectangles and rhombi that are thinner (ie longer) than the ones shown here have additional minimal spines that wind around the thin part.}  \label{rombi:fig}
\end{figure}

The subtler cases are those where $T$ has additional symmetries: in Fig.~\ref{rombi:fig} we show the minimal spines of $T$ as the metric varies from rectangular to square, fat rhombic, hexagonal, and finally thin rhombic. 

\subsection{Length.}
Let us define the \emph{length} of a union of curves on a flat torus
in the moduli space $\mathcal{M}$
as its normalized length, that is its one dimensional Hausdorff measure 
on a unit-area representative torus. 
For example, if you consider the explicit representative given by the parallelogram $\mathbb{R}$-generated by $1$ and $z\in\mathbb{C}$,
you have to divide lengths by the square root of the area 
of that parallelogram, that is $\sqrt{\Im(z)}$. 

Generalizing the well known concept for closed geodesics
on surfaces (see~\cite{Fm}), we define the
\emph{length spectrum of minimal spines} of a flat torus $T$ as
the set of lengths of its minimal spines.
Actually, we can define a more accurate spectrum, a set of
triples corresponding to the lengths of 
the sides of any spine on $T$, in decreasing order, that is
$$\left\lbrace \frac{(a,b,c)}{\sqrt{(b^2+c^2+bc)\Im(\tilde{p}(a,b,c))}}\,\Bigg|\,
(a,b,c)\in p^{-1}(T)\right\rbrace.$$
We have shown that both sets are finite, for any flat torus $T$.\\

In the following, we determine the spines of minimal length.

\begin{theorem}\label{mintheta}
Every unoriented flat torus has a unique spine 
of minimal length up to isometry.
In the oriented setting, instead, the same holds with the exception of the rectangular non square tori, for which there are exactly two.
\end{theorem}

\begin{proof}
Fix a torus $T\in\mathcal{M}^{\rm{no}}$. 
In the proof of Theorem~\ref{fibra}, we observed that all points of 
$\pi^{-1}(T)\cap H\subset \mathbb{H}^2$ have the same imaginary part,
so the lengths of the tripods we found with the map $\tilde{p}$
can be compared without need of normalization.
Clearly, only one is the shortest: that corresponding to $z_0\in D\cap \{\Re(z)\geq 0\}$.

In the oriented case, consider the representative
$z_0$ of $T\in\mathcal{M}$ with 
$z_0\in D \setminus \{z\in \partial D \, | \, \Re(z)<0\}$.
There are three cases:
\begin{itemize}
\item if $\Re(z_0)>0$, the unique shortest spine is
that associated to $z_0$ itself;
\item if $\Re(z_0)<0$, it is associated to $z_0+1$;
\item if $\Re(z_0)=0$, that is $T$ is a rectangular torus, both spines associated to
$z_0$ and $z_0+1$ are of minimal length.
These two spines are the same only for the square torus, indeed a rotation of $\frac{\pi}{2}$  sends one to the other.
\end{itemize}
These evident assertions can be verified by trigonometry, or by computing the lengths through the formula showed in the following .
\end{proof}

The explicit expression on $\tilde p(\mathcal{H})\subset\mathbb{H}^2$ for the \emph{length function} $L\colon \mathcal{H}\to (0,+\infty)$,
which to a spine assigns its length,
can be simply computed by the law of sines:
$$L(z) =\sqrt{ \frac{1+|z|^2-\Re(z)+\sqrt{3}\Im(z)}{\Im(z)}}.$$\\

The formula for the \emph{spine systole}
$\Spy\colon \mathcal{M}\to (0,+\infty)$, a generalized systole
which assigns to a flat torus the 
length of its shortest spines is 
$$\Spy(z)=\sqrt{ \frac{1+|z|^2-|\Re(z)|+\sqrt{3}\Im(z)}{\Im(z)} }. $$

Both functions
are proper and almost everywhere smooth. They extend to continuous functions
$L:\overline{\mathcal{H}}\to (0,+\infty] $ and
$\Spy:\overline{\mathcal{M}}\to (0,+\infty ]$.

For every continuous function $f :X\to\mathbb{R}$ on 
a topological $2$-manifold $X$, recall that
a point $p\in X$ is said
to be \emph{regular}  if there is a (topological) chart around
$p=0$ in which $f(x,y)=\textrm{const}+x$. 
Otherwise, it is \emph{critical}. 
A critical point
$p\in X$ is said to be \emph{non degenerate} if in a local (topological) chart around
$p=0$ we have either $f(x,y)=\textrm{const}-x^2+y^2$,
or $f(x,y)=\textrm{const}\pm(x^2+y^2)$.
The non degenerate
critical points are necessarily isolated. The function $f$ is \emph{topologically Morse} if it is proper and all
critical points are non degenerate. 
The classical Morse theory works also in the topological category.

\begin{remark}\rm{
The functions $L\colon  \mathcal{H}\to (0,+\infty)$ and $\Spy \colon \mathcal{M}\to (0,+\infty)$
are topologically Morse.
For both functions, the set of
sublevel $k$ is empty if $k< \sqrt[4]{3}\sqrt{2}$,
a point if $k=\sqrt[4]{3}\sqrt{2}$ and a topological disc otherwise.
Hence, the unique critical point is a minimum.
Therefore,
among all minimal spines on all flat tori, exactly one is the shortest. As expected,
it is the equilateral spine on the hexagonal torus. Its length is $\sqrt[4]{3}\sqrt{2}\approx 1.86$.
}\end{remark}


\section{Hyperbolic surfaces}
Let now $S_g$ be a closed orientable surface of genus $g\geq 2$. The oriented hyperbolic metrics on $S_g$ form the \emph{moduli space} $\calM(S_g)$ and the minimum length of a spine furnishes the \emph{spine systole}
$$\Spy \colon \calM(S_g) \longrightarrow \matR.$$
We now prove some facts on the function $\Spy$.

\begin{theorem}
The function $\Spy$ is continuous and proper. Its global minima are precisely the extremal surfaces.
\end{theorem}
\begin{proof}
The function $\Spy$ is clearly continuous because the length of spines varies continuously in the metric. We now prove properness as an easy consequence of the Collar Lemma \cite{Fm}. 

By Mumford's compactness theorem the subset $\calM_\varepsilon(S_g) \subset \calM(S_g)$ of all hyperbolic metrics with (closed geodesic) systole $\geq \varepsilon$ is compact for all $\varepsilon>0$, and the Collar Lemma says that for sufficiently small $\varepsilon>0$ a hyperbolic surface $S\in\calM(S_g) \setminus \calM_\varepsilon(S_g)$ has a simple closed curve $\gamma$ of length $<\varepsilon $ with a collar of diameter $C(\varepsilon)$, for some function $C$ such that $C(\varepsilon)\to \infty$ as $\varepsilon \to 0$. Every spine $\Gamma$ of $S_g$ must intersect $\gamma$ and cross the collar, hence $L(\Gamma) \geq C(\varepsilon)$ and therefore $\Spy$ is proper.

We now determine the global mimima of $\Spy$. Let $\Gamma$ be a spine in $S\in\calM(S_g)$ of minimal length. The spine has $6g-3$ edges. As in the Euclidean case, by cutting $S$ along $\Gamma$ we get a hyperbolic polygon $P$ with $12g-6$ edges and all interior angles $\frac {2\pi}3$. The length of $\Gamma$ is half the perimeter of $P$. 

Porti has shown \cite{Porti} that, among all hyperbolic $n$-gons with fixed interior angles, the one with smaller perimeter is the unique one that has an inscribed circle. Therefore among all polygons $P$ with angles $\frac{2\pi}3$ the one that minimizes the perimeter is precisely the regular one $R$, that is the one whose sides all have the same length. We deduce that the global minima for $\Spy$ are the hyperbolic surfaces that have $R$ as a fundamental domain, and these are precisely the extremal surfaces, as proved by Bavard \cite{Ba}.
\end{proof}

It would be interesting to investigate the function $\Spy$ and check for instance whether it is a topological Morse function, see \cite{Gend}.

In the flat case we have shown that the number of minimal spines is finite.
In the hyperbolic setting, we do not know if the same is true.
To conclude the section, we prove a partial result:

\begin{theorem}\label{finminspin}
The number of minimal spines with bounded length of a
closed hyperbolic surface $S$ is finite. 
\end{theorem}

\begin{proof}
In the proof of Theorem~\ref{existencein2d}, it results clear that 
every set of minimal spines of equibounded length is compact.
We now prove that for hyperbolic surfaces every such set
is discrete, hence finite.
By contradiction, let $\Gamma_n$ be a sequence of  distinct minimal spines of $S$ of
equibounded length, converging in the Hausdorff distance to the minimal spine $\Gamma$.
Moreover $L(\Gamma_n)\rightarrow L(\Gamma)$.
For every $\lambda\in [0,1]$ and every $n$ big enough, we construct, exactly as in the
proof of Theorem~\ref{minima},
the (not necessarily minimal) spine $\Gamma_n^\lambda$ and 
continuous function $F_n(\lambda)=L(\Gamma_n^\lambda)$.
The surface $S$ is hyperbolic, therefore, by Lemma~\ref{convexity},
$F_n(\lambda)$ is strictly convex.
Both $\Gamma$ and $\Gamma_n$ are minimal spines,
that is stationary points of the length functional, hence,
$F'_n(0)=F'_n(1)=0$ and $F_n$ is constant in $\lambda$:
a contradiction.
\end{proof}


\begin{thebibliography}{100}

\bibitem{Allard} \textsc{W. K. Allard}, \emph{On the first variation of a varifold},
Annals of Math. \textbf{95} (1972), 417--491.

\bibitem{AA} \textsc{W. K. Allard -- F. J. Almgren},
\emph{The Structure of Stationary One Dimensional Varifolds}.
Invent. Math. \textbf{34} (1976), 83--97.

\bibitem{AB} \textsc{S. B. Alexander -- R. L. Bishop},
\emph{Spines and Homology of Thin Riemannian Manifolds with Boundary},
Adv. Math. \textbf{155} (2000), 23--48.

\bibitem{AT} \textsc{L. Ambrosio -- P. Tilli}. ``Topics on Analysis in Metric Spaces,''
Oxford University Press, 2004.

\bibitem{AKP} \textsc{J. Aramayona -- T Koberda -- H. Parlier}, \emph{Injective maps between flip graphs}, {\tt arXiv:1409.7046}

\bibitem{A} \textsc{A. Ash},
\emph{Small-dimensional classifying spaces for arithmetic subgroups of general linear groups}, Duke Math. J. \textbf{51} (1984), 459--468.

\bibitem{Ba} \textsc{C. Bavard},
\emph{Disques extr\'emaux et surfaces modulaires}, Ann. Fac. Sci. Toulouse Math. \textbf{2} (1996), 191--202.

\bibitem{B} \textsc{M. A. Buchner}, \emph{Simplicial structure of the real analytic cut locus}, Proc. Amer. Math. Soc. \textbf{64} (1977), 118--121.

\bibitem{Choe} \textsc{J. Choe},
\emph{On the existence and regularity of fundamental domains with least boundary area},
J. Differential Geom. \textbf{29} (1989), 623--663.

\bibitem{Fm} \textsc{B.~Farb and D.~Margalit}
``A primer on mapping class groups,''
Princeton University Press, Princeton, NJ
    \textbf{49} (2012).

\bibitem{FST} \textsc{S. Fomin -- M. Shapiro -- D. Thurston}, \emph{Cluster algebras and triangulated
surfaces. I. Cluster complexes}, Acta Math. \textbf{201} (2008), 83--146.

\bibitem{Gend} \textsc{M. Gendulphe}, \emph{The injectivity radius of hyperbolic surfaces and some Morse functions over moduli spaces}, {\tt arXiv:1510.02581}

\bibitem{GG} \textsc{E. Girondo -- G. Gonz´alez-Diez}, \emph{On extremal Riemann surfaces and their uniformizing Fuchsian groups}, Glasgow Math. J. \textbf{44} (2002), 149--157.

\bibitem{H} \textsc{M. H. Hirsch}, \emph{Smooth regular neighborhoods}, Ann. Math. \textbf{76} (1962), 524--530.

\bibitem{KP} \textsc{M. Korkmaz -- A. Papadopoulos}, \emph{On the ideal triangulation graph of a
punctured surface}, Ann. Inst. Fourier \textbf{62} (2012), 1367--1382.

\bibitem{maggi} \textsc{F. Maggi}  \emph{Sets of Finite Perimeter and Geometric Variational Problems. An introduction to geometric measure theory.}
Cambridge University Press, Cambridge (2012). 

\bibitem{M} \textsc{B. Martelli}, \emph{Complexity of PL manifolds}, Alg. \& Geom. Top. \textbf{10} (2010), 1107--1164.

\bibitem{Matveev-moves} \textsc{S. Matveev}, \emph{Transformations of special spines and the Zeeman conjecture}, Math. USSR-Izv. \textbf{31} (1988), 423--434.

\bibitem{Matveev} \bysame, \emph{Complexity theory of three-dimensional manifolds}, Acta Appl. Math. \textbf{19} (1990), 101--130.

\bibitem{Matveev-book} \bysame, ``Algorithmic Topology and Classification of 3-Manifolds'', second edition, Algorithms and Computation in Math. \textbf{9}, Springer, Berlin (2007).

\bibitem{Piergallini} \textsc{R. Piergallini}, \emph{Standard moves for standard polyhedra and spines}, Rendiconti Circ. Mat. Palermo \textbf{37} (1988), 391--414.

\bibitem{Porti} \textsc{J. Porti}, \emph{Hyperbolic polygons of minimal perimeter with given angles}, Geom. Dedicata \textbf{156} (2012), 165--170. 

\bibitem{STT} \textsc{D. Sleator -- D. D. Tarjan -- W. P. Thurston}, \emph{Rotation distance, triangulations, and hyperbolic geometry}, J. Amer. Math. Soc. \textbf{1} (1988), 647--681.

\end{thebibliography}
\end{document}